\documentclass[twoside]{amsart}
\usepackage{amsmath,amssymb,amscd,amsthm,amsxtra,amsfonts}
\usepackage{mathtools,mathrsfs,dsfont,xparse}
\usepackage{graphicx,epstopdf}
\usepackage{multirow}
\usepackage{cases}
\usepackage{enumerate}
\usepackage{xcolor}
\usepackage{tikz,pgfplots}
\usetikzlibrary{matrix}
\usepackage{mathtools}
\usepackage[margin=0.9 in]{geometry}
\newtheorem{theorem}{Theorem}[section]
\newtheorem{lemma}[theorem]{Lemma}
\newtheorem{proposition}[theorem]{Proposition}

\newcommand{\R}{\mathbb{R}}

\newcommand{\f}{\frac}
\newcommand{\beq}{\begin{equation}}
\newcommand{\eeq}{\end{equation}}
\newcommand{\beqq}{\begin{equation*}}
\newcommand{\eeqq}{\end{equation*}}

\theoremstyle{definition}

\theoremstyle{remark}
\newtheorem{remark}[theorem]{Remark}

\numberwithin{equation}{section}

\usepackage[colorlinks=true, pdfstartview=FitV, linkcolor=blue, citecolor=blue, urlcolor=blue]{hyperref}

\DeclareMathOperator{\re}{Re}
\DeclareMathOperator{\im}{Im}

\def\T{{\Bbb{T}}}

\allowdisplaybreaks[4]

\DeclareDocumentCommand{\abs}{s m}{
  \operatorname{}
  \IfBooleanTF{#1}{#2}{\left|#2\right|}}

\DeclareDocumentCommand{\norm}{s m}{
  \operatorname{}
  \IfBooleanTF{#1}{#2} {\left\| #2\right\|}}

\DeclareDocumentCommand{\inner}{s m}{
  \operatorname{}
  \IfBooleanTF{#1}{#2}{\left \langle#2\right \rangle}}

\DeclareDocumentCommand{\parenthese}{s m}{
  \operatorname{}
  \IfBooleanTF{#1}{#2}{\left(#2\right)}}

\DeclareDocumentCommand{\square}{s m}{
  \operatorname{}
  \IfBooleanTF{#1} {#2}{\left[#2\right]}}

\DeclareDocumentCommand{\bracket}{s m}{
  \operatorname{}
  \IfBooleanTF{#1}{#2}{\left\{#2\right\}}}

\numberwithin{equation}{section}
\mathtoolsset{showonlyrefs}

\begin{document}

\address{Xueying  Yu
\newline \indent Department of Mathematics, University of Washington\indent 
\newline \indent  C138 Padelford Hall Box 354350, Seattle, WA 98195,\indent }
\email{xueyingy@uw.edu}

\address{Haitian Yue
\newline \indent Institute of Mathematical Sciences, ShanghaiTech University\newline\indent
Pudong, Shanghai, China.}
\email{yuehaitian@shanghaitech.edu.cn}

\address{Zehua Zhao
\newline \indent Department of Mathematics and Statistics, Beijing Institute of Technology,
\newline \indent MIIT Key Laboratory of Mathematical Theory and Computation in Information Security,
\newline \indent  Beijing, China. \indent}
\email{zzh@bit.edu.cn}

\title[Fourth-order Schr\"odingers on waveguide manifolds]{Global Well-posedness and scattering for fourth-order Schr\"odinger equations on waveguide manifolds}
\author{Xueying Yu, Haitian Yue and Zehua Zhao}
\maketitle

\begin{abstract}
In this paper,  we study the well-posedness theory and the scattering asymptotics for fourth-order Schr\"odinger equations (4NLS) on waveguide manifolds (semiperiodic spaces) $\mathbb{R}^d\times \mathbb{T}^n$, $d \geq 5$, $n=1,2,3$. The tori component $\mathbb{T}^n$ can be generalized to $n$-dimensional compact manifolds $\mathcal{M}^n$.  First,  we modify Strichartz estimates for 4NLS on waveguide manifolds, with which we establish the well-posedness theory in proper function spaces via the standard contraction mapping method.  Moreover, we prove the scattering asymptotics based on an interaction Morawetz-type estimate established for 4NLS on waveguides. At last, we discuss the higher dimensional analogue, the focusing scenario and give some further remarks on this research line. This result can be regarded as the waveguide analogue of Pausader \cite{Pau2,Pau1,Pau3} and the 4NLS analogue of Tzvetkov-Visciglia \cite{TV2}.

\bigskip

\bigskip

\end{abstract}

\setcounter{tocdepth}{1}
\tableofcontents

\parindent = 10pt     
\parskip = 8pt

\section{Introduction}\label{sec Intro}
\subsection{Statement of main results}
In this paper, we study the defocuing, fourth-order Schr\"odinger equations (4NLS) on waveguide manifolds $\mathbb{R}^d\times \mathbb{T}$ in the energy space  $H^2(\mathbb{R}^d\times \mathbb{T})$,
\begin{align}\label{maineq}
\begin{cases}
 (i\partial_t+\Delta_{x,\alpha}^{2})u+|u|^pu=0,\\
  u(0)= u_0(x,\alpha)\in H^2(\mathbb{R}^d\times \mathbb{T}),
\end{cases}
\end{align}
with $(t,x,\alpha)\in \mathbb{R}_t \times \mathbb{R}^d \times \mathbb{T}$, where $d\geq 5$ and $\frac{8}{d}<p<\frac{8}{d-3}$. 

Here the space $\mathbb{R}^d \times \mathbb{T}$ is a special case of the product spaces $\mathbb{R}^{d} \times \mathbb{T}^{n}$, which is known as `semiperiodic spaces' as well as `waveguide manifolds' or `waveguides', where $\mathbb{T}^{n}$ is a (rational or irrational) $n$-dimensional torus. Moreover, $\Delta_{x,\alpha}=\Delta_{x}+\Delta_{\alpha}$ with $\Delta_{\alpha}$ the Laplace–Beltrami operator on $\mathbb{T}$ and $\Delta_{x}=\sum_{i=1}^{d}\partial^2_{x_i}$ is the Laplace operator associated to the flat metric on $\mathbb{R}^d$.
\begin{remark}
We note that the exponent $p$ is in the subcritical range for technical reasons. The left endpoint indicates the mass-critical setting if we ignore the torus-direction; while the right endpoint indicates the energy-critical scenario if we treat the torus-direction as additional Euclidean-direction. So essentially, the problem is both energy-subcritical and mass-supercritical. 
\end{remark}
\begin{remark}
As mentioned in the abstract, we also cover the higher dimensional case ($\mathbb{R}^d\times \mathbb{T}^n$, $d \geq 5$, $n=1,2,3$). (See Theorem \ref{main2} and Section \ref{sec Discussion}). The proofs can be generalized in a natural way so we mainly consider \eqref{maineq} as a model case.
\end{remark}
\begin{remark}
The torus component $\mathbb{T}$ can be generalized to $\mathcal{M}$, a compact Riemannian manifold  without boundary of dimension one  (in other words, we consider more general manifolds $\mathbb{R}^d\times \mathcal{M}$ instead of $\mathbb{R}^d\times \mathbb{T}$ in \eqref{maineq}) and the following discussions and results will still hold.
\end{remark}
For 4NLS \eqref{maineq}, like the classical NLS, the following quantities are conserved:
\begin{align*}
\text{Mass: }    &\quad
{M}(u(t))  = \int_{\mathbb{R}^d\times \mathbb{T}} |u(t,x,\alpha)|^2\, d x d\alpha,\\
\text{   Energy:  }    &  \quad
{E}(u(t))  = \int_{\mathbb{R}^d\times \mathbb{T}} \frac12 |\Delta_{x,\alpha} u(t,x,\alpha)|^2  + \frac{1}{p+2} |u(t,x,\alpha)|^{p+2} \, dx d\alpha.
\end{align*}
We intend to prove the global well-posedness and scattering for \eqref{maineq}. One of the main results is as follows.

\begin{theorem}\label{main}
The initial value problem \eqref{maineq} has a unique local solution $u(t,x,\alpha)\in \mathcal{C}((-T,T);H_{x,\alpha}^{2} (\mathbb{R}^d\times \mathbb{T}))$ where $T=T(\|u_0\|_{H_{x,\alpha}^{2}}(\mathbb{R}^d\times \mathbb{T}))>0$; moreover, the solution $u(t, x, \alpha)$ can be extended globally in time by iteration. Furthermore it scatters in the energy space, that is, there exist
$f^{\pm} \in H^2_{x,\alpha}(\mathbb{R}^d\times \mathbb{T})$ such that
\begin{equation*}
    \lim_{t \rightarrow \pm \infty} \|u(t,x,\alpha)-e^{it\Delta_{x,\alpha}^2}f^{\pm}\|_{H^2_{x,\alpha} (\mathbb{R}^d\times \mathbb{T})} =0.
\end{equation*}
\end{theorem}
\begin{remark}
Theorem \ref{main} can be regarded as the 4NLS analogue of Theorem 1.1 and Theorem 1.3 in Tzvetkov-Visciglia \cite{TV2}. 
\end{remark}
\begin{remark}
We note that the required range $\frac{8}{d}<p<\frac{8}{d-3}$ is for the sake of scattering. If one only considers the well-posedness theory, the result can be extended to the case $0<p<\frac{8}{d}$ with little modifications. See \cite{TV2} for the NLS case.
\end{remark}
In fact, we can deal with more general settings (on waveguides with higher torus-dimensions) as follows,
\begin{align}\label{maineq2}
\begin{cases}
  (i\partial_t+\Delta_{x,\alpha}^{2})u+|u|^pu=0,\\
   u(0)= u_0(x,\alpha)\in H^2(\mathbb{R}^d\times \mathbb{T}^n),
\end{cases}
\end{align}
where $d\geq 5$ and $\frac{8}{d}<p<\frac{8}{d+n-4}$ ($n=2,3$).
And the corresponding result is
\begin{theorem}\label{main2}
The initial value problem \eqref{maineq2} has a unique local solution $u(t,x,\alpha)\in \mathcal{C}((-T,T);H_{x,\alpha}^{2} (\mathbb{R}^d\times \mathbb{T}^n))$ where $T=T(\|u_0\|_{H_{x,\alpha}^{2}}(\mathbb{R}^d\times \mathbb{T}^n))>0$; moreover, the solution $u(t, x, \alpha)$ can be extended globally in time by iteration. Furthermore it scattering in the energy space: , that is, there exist
$f^{\pm} \in H^2_{x,\alpha} (\mathbb{R}^d\times \mathbb{T}^n)$ such that
\begin{equation*}
    \lim_{t \rightarrow \pm \infty} \|u(t,x,\alpha)-e^{it\Delta_{x,\alpha}^2}f^{\pm}\|_{H^2_{x,\alpha} (\mathbb{R}^d\times \mathbb{T}^n)} =0.
\end{equation*}
\end{theorem}
\begin{remark}
Similarly to \eqref{maineq}, the exponent $p$ in \eqref{maineq2} also corresponds to the energy-subcritical and mass-supercritical range.
\end{remark}
\begin{remark}
We can see that in \eqref{maineq2} the torus-dimension is at most three. There are two  reasons. First, if $n=4$, there is no $p$ satisfying the condition $\frac{8}{d}<p<\frac{8}{d+n-4}$. At the end point case  $p=\frac{8}{d}$, the problem is both mass-critical and energy-critical which requires more techniques. See \cite{HP,Z1} for the NLS case. Second, from the aspect of Sobolev embedding (technical reason), one can see that the $L^{\infty}$-norm can be controlled by the $H^2$-norm (consistent with the initial energy space) if and only if the dimension is less than four. 
\end{remark}
\begin{remark}
Theorem \ref{main} and Theorem \ref{main2} are in the defocusing setting. In Section \ref{sec Discussion}, we will discuss the focusing scenario under the mass-subcritical setting. One may consider more general cases. In general, the dynamics of focusing equations are rich and interesting.
\end{remark}
\begin{remark}
Regarding Theorem \ref{main} and Theorem \ref{main2}, there are some related problems which are natural to raise, for example, the critical case when $p=\frac{8}{d}$ or $\frac{8}{d+n-4}$. We will discuss more in  Appendix \ref{sec Apx}.
\end{remark}
\subsection{Background}
The soliton instabilities (e.g. the finite time blowup) of nonlinear waves, which can be characterized by nonlinear Schr\"odinger equations in many canonical physical models, is one of the basic phenomena in nonlinear physics. 
It is well-known that the stability of the soliton instabilities depends on the number of space dimensions and strength of nonlinearity.
On the other hand, fourth-order Schr\"odinger equations have been introduced by Karpman \cite{Karpman} and Karpman and Shagalov \cite{KS} to investigate the stabilizing role of higher-order dispersive
effects for the soliton instabilities.
The following work \cite{FIP} by Fibich, Ilan and Papanicolaou studies the  self-focusing (i.e. finite time blowup in dimension two) and singularity formation of such fourth-order Schr\"odinger equations  from the mathematical viewpoint.
More precursory research on the basic properties of 4NLS can be found in \cite{BKS,GW2002,HHW1, HHW2, Segata}.
It is worth to mention that the defocusing energy-critical 4NLS in dimension eight was first proved in the series of work by Pausader \cite{Pau2, Pau1}
and then the higher dimension cases ($d\geq 9$) are handled by Miao, Xu and Zhao \cite{MXZ2}.
For more developments of 4NLS, we refer to 
\cite{MWZ,MXZ1, MZ,Pau3,PS,Zheng, PX} in the Euclidean space,
 \cite{CHKL,Dinh,GSWZ,HS, OW, Kwak} on the tori space and the references therein. 
Furthermore, the recent works  \cite{OS,OST,OT,OTW} studied the long behavior of 4NLS from the probabilistic viewpoint.

In today's backbone networks, data signals are almost exclusively transmitted by optical carriers in fibers.  Applications like the internet demand an increase of the available bandwidth in the network and reduction of costs for the transmission of data. Product spaces $\R^m \times \T^{n}$ are known as ‘waveguide manifolds’ and are of particular interest in nonlinear optics of telecommunications. 
Well-posedness theory and the long time dynamics for 4NLS on waveguide manifolds are understudied. Before we go into 4NLS on waveguide manifolds,
let us review some results of NLS in the waveguide manifold setting.  Generally, well-posedness theory and long time behavior of NLS are hot topics in the area of dispersive evolution equations and have been studied widely in recent decades. Naturally, the Euclidean case is first treated and the theory, at least in the defocusing setting, has been well established. We refer to \cite{Iteam1,BD3,KM1} for some important  Euclidean results. Moreover, we refer to \cite{CZZ,CGZ,HP,HTT1,HTT2,IPT3,IPRT3,KV1,YYZ,Z1,Z2,ZhaoZheng} with regard to the torus and waveguide settings.
  We may roughly think of the waveguide case as the ``intermediate point" between the Euclidean case and the torus case since the waveguide manifold is a product of  Euclidean spaces and the tori. The techniques used in Euclidean and torus settings are frequently combined and applied to the waveguides problems.
One of the most interesting long time behaviers of NLS in the waveguide spaces $\mathbb{R}^d\times \mathbb{T}^n$ is that the global solutions exhibit the different asymptotic behaviers: the \emph{scattering} \cite{CGYZ, Z1, Z2}  and \emph{modified scattering} \cite{ZPTV, Liu}, which means the solution asymptotical behaves like the resonant system based on the frequencies from the torus component $\mathbb{T}^n$,  depending on the different dimensions of the Euclidean-direction and of the torus-direction in $\mathbb{R}^d\times \mathbb{T}^n$. 
 To the authors' best knowledge, the current paper is the first scattering result towards understanding long time dynamics for the 4NLS within the context of waveguides.

Typically, the most important ingredient in a scattering argument is the Morawetz estimates. They  are {\it a priori} monotonicity formulas that hold for solutions to dispersive equations, which control the long-time behavior of solutions. For example, the Morawetz estimate for NLS on $\R^3$ in \cite{Iteam1} reads 
\begin{align*}
\norm{u}_{L_{t,x}^4 (\R \times \R^3)}^4 \lesssim \norm{u}_{L_{t}^{\infty} L_{x}^{2} (\R \times \R^3)}^3   \norm{u}_{L_{t}^{\infty} \dot{H}_{x}^{1}(\R \times \R^3) } . 
\end{align*}
Roughly speaking, such estimate suggests that certain $L_x^p$ norm of the solution decays in time, which implies the scattering effect. More Morawetz estimate results for NLS  can be found in \cite{CGT,PV,RV, V} for Euclidean spaces and \cite{TV2} on the waveguide setting. It is worth mentioning that to \cite{CGT,PV} both deal with the Morawetz estimates in low dimensions, and \cite{PV} relied on bilinear virial identities  while \cite{CGT} used  a tensor product method. In the partially periodic setting, the work \cite{TV2} modified the weight function in the Morawetz action by removing its dependence on the periodic direction and obtained a  suitable version  of the interaction Morawetz estimates. Let us also mention that \cite{TV2} used the available low dimensional Morawetz estimates in \cite{PV} as a black box in their partially periodic calculation. In fact,  to be able to make use of the Euclidean  result, all the terms in Morawetz computation along the periodic direction vanish after integration by parts. This benefits from   their well-chosen weight (independent of the periodic direction) and no mixed derivatives in Laplacian.

In the proof of the scattering result in Theorem \ref{main}, we need a suitable  interaction Morawetz-type estimate for 4NLS on $\R^m \times \T$ to close the scattering argument. What is known in the 4NLS context,  \cite{Pau2} proved  an interaction Morawetz estimate following previous analysis from \cite{Iteam1} for dimension $d\geq 7$ and  \cite{MWZ} extended the range of the dimension of the interactive estimate  to $d \geq 5$ in \cite{Pau2} by modifying a tensor product method appeared in  \cite{CGT}. Inspired by  \cite{TV2}, we wish to modify the weight function in the Morawetz action and use the low dimensional  interaction Morawetz estimates in \cite{MWZ} as a black box. However, after integration by parts, all the terms in Morawetz calculation on the periodic direction  would not die out, and  even worst, they mix with the Euclidean directions. This is simply because  the bi-Laplacian,  unlike Laplacian, contains mixed derivatives (for example in two dimensions $\Delta^2 = \partial_{x_1 x_1 x_1 x_1} + 2 \partial_{x_1 x_1 x_2 x_2} + \partial_{x_2 x_2 x_2 x_2}$, while $\Delta =\partial_{x_1 x_1 } + \partial_{x_2 x_2 }  $). Such terms ($\partial_{x_1 x_1 x_2 x_2}$) causes great difficulty in reducing the  partially periodic setting to its corresponding Euclidean setting. Therefore, we first `open the box', and use the tensor product method introduced in  \cite{CGT}  to compute all the mixed/non-mixed terms in the Morawetz action. Then since it is impossible to make  mixed terms disappear, we  carry  them all the way to the end of the calculation and hope that they could still give the right signs in the Morawetz inequality. Fortunately, it turns out that they do have the right signs, hence we manage to derive a Morawetz type estimate working for our setting.

\subsection{Organization of the rest of this paper}
In Section \ref{sec Pre}, we discuss some useful estimates and   fixed admissible exponents; in Section \ref{sec LWP}, we establish the well-posedness theory; in Section \ref{sec Morawetz}, we prove a Morawetz-type estimate which is the key ingredient to show the decay property of the solution; in Section \ref{sec Scattering}, we show the decay property and then use it to prove the scattering behavior; in Section \ref{sec Discussion}, we discuss the higher dimensional analogue and the focusing scenario; in Section \ref{sec Remarks}, we make a few more remarks on this research line; in Appendix \ref{sec Apx}, we include the decay property for NLS and 4NLS on general waveguide manifolds.

\subsection{Notations}
Throughout this paper, we use $C$ to denote  the universal constant and $C$ may change line by line. We say $A\lesssim B$, if $A\leq CB$. We say $A\sim B$ if $A\lesssim B$ and $B\lesssim A$. We also use notation $C_{B}$ to denote a constant depends on $B$. We use usual $L^{p}$ spaces and Sobolev spaces $H^{s}$. Moreover, we write $\inner{x} = (1+|x|^2)^{\f{1}{2}}$, and $p'$ for the dual index of $p \in (1,+\infty)$ in the sense that $\frac{1}{p^{'}}+\frac{1}{p}=1$. 

We regularly refer to the composed spacetime norms
\begin{equation}
    ||u||_{L^p_tL^q_z(I_t \times \mathbb{R}^m\times \mathbb{T}^n)}=\left(\int_{I_t}\left(\int_{\mathbb{R}^m\times \mathbb{T}^n} |u(t,z)|^q \, dz \right)^{\frac{p}{q}}  \, dt\right)^{\frac{1}{p}}.
\end{equation}

Similarly we can define the composition of three $L^p$-type norms like $L^p_tL^q_xL^2_{\alpha}$.

\section{Useful estimates and admissible exponents}\label{sec Pre}
In this section, we discuss some useful estimates (mainly Strichartz-type estimates in the setting of waveguides) which are fundamental for both of well-posedness and scattering theory. Moreover, we fix some admissible exponents for future use. 

We overview the Strichartz estimates for 4NLS in the Euclidean setting first. We refer to \cite{Pau2,Pau1} for more details. (See \cite{Taobook} for the classical NLS case.) There are two versions which corresponds to two types of admissible pairs ($S$-admissible and $B$-admissible). We state them respectively.

We say that $(p,q)$ is S-admissible if
\begin{equation}
    \frac{2}{p}+\frac{d}{q}=\frac{d}{2},\quad 2 \leq p,q \leq \infty \quad (p,q,d)\neq (2,\infty,2).
\end{equation}
We define the Strichartz norm by
\begin{equation}
    \|u\|_{S^{s}_{p,q}}:=\|u\|_{L^p_{t\in I}W_{x}^{s,q}}
\end{equation}
where $I=[0,T)$. Then one version of Strichartz estimates for 4NLS reads
\begin{lemma}[Strichartz estimate]
For S-admissible pairs $(p,q)$ and $(a,b)$, we
have
\begin{equation}
   \|e^{it\Delta_{x}^{2}} u_0\|_{S^{s}_{p,q}} \lesssim \| |\nabla|^{-\frac{2}{p}} u_0\|_{H_x^s}
\end{equation}
and
\begin{align}
\norm{\int_0^t e^{i(t-s) \Delta_{x}^2} F(s) \, ds}_{S^s_{p,q}} \lesssim \norm{ \abs{\nabla}^{s-\frac{2}{p} - \frac{2}{a}} F}_{L_{t \in I}^{a'} L_x^{b'}}
\end{align}
\end{lemma}
\begin{remark}
If we take the $\alpha$-direction (torus-component) into consideration, we expect to have
\begin{equation}
   \|e^{it\Delta_{x,\alpha}^{2}} u_0\|_{S^{s}_{p,q}H_{\alpha}^{\gamma}} \lesssim \| |\nabla|^{-\frac{2}{p}} u_0\|_{H_x^sH_{\alpha}^{\gamma}}
\end{equation}
and
\begin{equation}
    \norm{\int_0^{t}e^{i(t-s)\Delta_{x,\alpha}^{2}} F(s) \, ds}_{S^{s}_{p,q}H_{\alpha}^{\gamma}} \lesssim \norm{|\nabla|^{s-\frac{2}{p}-\frac{2}{a}}F}_{L_{t\in I}^{a^{'}}L_x^{b^{'}}H_{\alpha}^{\gamma}}.
\end{equation}
We do not give the proof here since we will not use them. We leave it for interested readers.
\end{remark}
If we consider B-admissible pair with $s$ regularity in the sense of ($\frac{4}{p}+\frac{d}{q}=\frac{d}{2}-s$), another version of Strichartz estimate reads
\begin{lemma}
\begin{equation}
   \|e^{it\Delta_{x}^{2}} u_0\|_{L_t^pL_x^q} \lesssim \|  u_0\|_{H_x^s}
\end{equation}
and
\begin{equation}
    \norm{\int_0^{t}e^{i(t-s)\Delta_{x}^{2}} F(s) \, ds}_{L_t^pL_x^q} \lesssim \| |\nabla|^s F\|_{L_{t\in I}^{a^{'}}L_x^{b^{'}}}.
\end{equation}
\end{lemma}
For convenience, we will apply the $B$-admissible version of Strichartz estimate. Moreover, if we take the $\alpha$-direction (torus-component) into consideration, we have
\begin{lemma}\label{Strichartz}
\begin{equation}
   \|e^{it\Delta_{x,\alpha}^{2}} u_0\|_{L_t^pL_x^qH_{\alpha}^{\gamma}} \lesssim \|  u_0\|_{H_x^sH_{\alpha}^{\gamma}}
\end{equation}
and
\begin{equation}
    \norm{\int_0^{t}e^{i(t-s)\Delta_{x,\alpha}^{2}} F(s) \, ds}_{L_t^pL_x^qH_{\alpha}^{\gamma}} \lesssim \| |\nabla|^s F\|_{L_{t\in I}^{a^{'}}L_x^{b^{'}}H_{\alpha}^{\gamma}}.
\end{equation}
\end{lemma}
\begin{proof}
The proof of Lemma \ref{Strichartz} is very similar to Proposition 2.1 of \cite{TV1}, which is the classical NLS case, so we just explain the difference from the NLS proof here. For  NLS, the main idea is decomposing the functions with respect to the orthonormal basis of $L^2(\mathbb{T})$ given by the eigenfunctions $\{\phi_j\}_j$ of $-\Delta_\alpha$. For 4NLS, we consider $\phi_j=\phi_j(\alpha)$ then
\begin{equation}
    (-\Delta_x-\Delta_{\alpha})^2\phi_j=\lambda_j^2\phi_j,\quad \lambda_j>0.
\end{equation}
So we can write $u(x,\alpha)$ as
\begin{equation}
    u(x,\alpha)=\sum_{j}u_j(x)\phi_j(\alpha),
\end{equation}
and $u_j(x)$ in the Fourier space satisfies
\begin{equation}
i\partial_t\hat{u}_j+(|\xi|^2+\lambda_j^2)^{2}\hat{u}_j=\hat{F}_j.    
\end{equation}
Hence we are reduced to Guo-Wang \cite{GW} for $(|\xi|^2+\lambda_j^2)^{2}$. And the estimate will follow in a standard way.
\end{proof}

\begin{remark}
Similar to \cite{TV1}, for the estimates above, the tori can be generalized to $k$-dimensional compact Riemannian manifold. This allows one to deal with the case when the torus dimension is higher than one (see Section \ref{sec Discussion}).
\end{remark}

Then we recall a useful lemma as follows (see \cite{TV2})
\begin{lemma}\label{delta}
For every $0<s<1$, $p>0$ there exists $C = C(p, s)>0$ such that
\begin{equation}
    \||u|^pu\|_{\dot{H}_{\alpha}^s} \leq C \|u\|_{\dot{H}_{\alpha}^s}\|u\|^p_{L_{\alpha}^{\infty}}.
\end{equation}
\end{lemma}
Next we discuss the admissible exponents which will be used in the following sections. The `double subcritical' natural gives us enough room to establish well-posedness theory and prove the scattering. We refer to Section \ref{sec LWP} for the well-posedness part and Section \ref{sec Scattering} for the scattering for more details. Moreover, one may compare this   with Section 3 and Section 6 in \cite{TV2} where NLS was considered.

We summary the indices in the following two lemmas. The first one will be used in the proof of  well-posedness theory in Section \ref{sec LWP}  while the second one will be used in the proof of scattering in Section \ref{sec Scattering}.
\begin{lemma}\label{index1}
Consider $s>\frac{1}{2}$ and $\delta>0$ satisfy
\begin{equation}
   s+\frac{1}{2}+\delta \leq 2.
\end{equation}
One can find B-admissible pair $(l,m)$, B-admissible pair (with $s$ regularity) $(q,r)$ and dual B-admissible (with $s$-regularity) $(\tilde{q},\tilde{r})$ satisfy
\begin{equation}
    \frac{1}{\tilde{r}^{'}}=\frac{p+1}{r},\frac{1}{\tilde{q}^{'}}>\frac{p+1}{q},
\end{equation}
\begin{equation}
    \frac{1}{m^{'}}=\frac{1}{m}+\frac{p}{r},\quad \frac{1}{l^{'}}>\frac{1}{l}+\frac{p}{q},
\end{equation}
and
\begin{equation}
    2<\frac{mp}{m-2}<\frac{2d}{d-3}.
\end{equation}
\end{lemma}

\begin{lemma}\label{index2}
Consider $s>\frac{1}{2}$ and $\delta>0$ satisfy
\begin{equation}
   s+\frac{1}{2}+\delta \leq 2.
\end{equation}
Then one can find B-admissible (with $s$-regularity) indices $(q_{\theta},r_{\theta})$, dual B-admissible (with $s$-regularity) $(\tilde{q}_{\theta},\tilde{r}_{\theta})$ and B-admissible indices $(l,m)$ satisfy the following relations.
\begin{equation}
    \frac{4}{q_{\theta}}+\frac{d}{r_{\theta}}=\frac{d}{2}-s,\frac{4}{q_{\theta}}+\frac{d}{\tilde{r}_{\theta}}+\frac{4}{\tilde{q}_{\theta}}+\frac{d}{r_{\theta}}=d,
\end{equation}
\begin{equation}
    \frac{1}{(p+1)\tilde{q_{\theta}}^{'}}=\frac{\theta}{q_{\theta}}, \frac{1}{(p+1)\tilde{r_{\theta}}^{'}}=\frac{\theta}{r_{\theta}}+\frac{2(1-\theta)}{pd},
\end{equation}
and
\begin{equation}
    \frac{4}{l}+\frac{d}{m}=\frac{d}{2},\frac{1}{m^{'}}=\frac{1}{m}+\frac{p}{r_{\theta}},\frac{1}{l^{'}}=\frac{1}{l}+\frac{p}{q_{\theta}}.
\end{equation}
\end{lemma}
We omit the proofs since they are straightforward.

\section{Well-posedness theory}\label{sec LWP}
In this section, we establish local well-posedness theory for the 4NLS in Theorem \ref{main} by the standard contraction mapping method. Then together with the conservation law, we extend the solution to global. The main works are constructing proper function spaces which allows one to show that the natural Duhamel mapping is a contraction mapping. It is tightly based on Strichartz estimates for 4NLS on $\mathbb{R}^d\times \mathbb{T}$  and careful choices of the exponents. We refer to Section 4 of \cite{TV2} for the NLS analogue.\vspace{3mm}

First, we introduce the integral operator by Duhamel formula,
\begin{equation}
    \Phi_{u_0}(u) :=e^{it\Delta_{x,\alpha}^{2}} u_0+i\int_0^{t}e^{i(t-s)\Delta_{x,\alpha}^{2}} (|u|^{p}u) \, ds.
\end{equation}
For convenience, we will only write $\Delta$ instead of $\Delta_{x,\alpha}$ all along this section. We then define the following three norms,
\begin{align*}
 \|u\|_{X_T} & :=\|u\|_{L_t^qL_x^rH_{\alpha}^{\frac{1}{2}+\delta}([-T,T]\times \mathbb{R}^d \times \mathbb{T})},\\
  \|u\|_{Y_T^1} & :=\sum_{k=0,1,2}\|\nabla^k_x u\|_{L^l_tL^m_xL^2_{\alpha}([-T,T]\times \mathbb{R}^d \times \mathbb{T})},\\
\|u\|_{Y_T^2} &:=\sum_{k=0,1,2}\|\partial^k_{\alpha} u\|_{L^l_tL^m_xL^2_{\alpha}([-T,T]\times \mathbb{R}^d \times \mathbb{T})},
\end{align*}
where $(l,m)$ is B-admissible and $(q,r)$ is B-admissible with $s$ regularity ($s+\frac{1}{2}+\delta \leq 2$) which are chosen according to Lemma \ref{index1} and Lemma \ref{index2}. Combining them together, we define
\begin{equation}
    \|u\|_{Z_T} :=\|u\|_{X_T}+\|u\|_{Y_T^1}+\|u\|_{Y_T^2}.
\end{equation}
Also, for convenience, we will omit writing $([-T,T]\times \mathbb{R}^d \times \mathbb{T})$ all along this section. Now we show that the Duhamel operator $\Phi_{u_0}$ is a contraction mapping. The following steps are standard, which are corresponding to those in Section 4 of \cite{TV2}.\vspace{3mm}

\textbf{Step 1. ($\Phi_{u_0}$ is from $Z$ to $Z$)} For all $u_0 \in H^2_{x,\alpha}$, there exist $T=T(\|u_0\|_{H^2_{x,\alpha}}>0$ and $R=R(\|u_0\|_{H^2_{x,\alpha}})>0$ such that $\Phi_{u_0}(B_{Z_{T^{'}}}) \subset B_{Z_{T^{'}}}$, for any $T^{'}<T$.

Consider the $X$ norm first. By Strichartz estimates, Lemma \ref{delta} and the H\"older inequality,
\begin{equation}
\| |u|^p u \|_{L_t^{\tilde{q}^{'}}L_x^{\tilde{r}^{'}}H_{\alpha}^{\frac{1}{2}+\delta}} \lesssim \| \|u\|^{p+1}_{H_{\alpha}^{\frac{1}{2}+\delta}} \|_{L_t^{\tilde{q}^{'}}L_x^{\tilde{r}^{'}}} \lesssim T^{\beta(p)}\|u\|^{p+1}_{L_t^qL_x^rH_{\alpha}^{\frac{1}{2}+\delta}},
\end{equation}
with $\beta(p)>0$. Here we choose the indices such that,
\begin{equation}
    \frac{1}{\tilde{r}^{'}}=\frac{p+1}{r},\frac{1}{\tilde{q}^{'}}>\frac{p+1}{q}.
\end{equation}
It is manageable since the problem is subcritical, which is also similar to the NLS case. See Lemma \ref{index1}.\vspace{3mm}

Consider $Y^1$ and $Y^2$ norms. By Strichartz estimate and the H\"older inequality,
\begin{equation}
    \aligned
    \|D^k (|u|^p u)\|_{L^{l^{'}}_tL^{m^{'}}_xL^{2}_{\alpha}} & \lesssim \| \|D^k u\|_{L^{2}_{\alpha}} \|u\|^p_{L_{\alpha}^{\infty}}\|_{L^{l^{'}}_tL^{m^{'}}_x} \\
    &\lesssim \| \|D^k u\|_{L^{2}_{\alpha}} \|u\|^p_{H_{\alpha}^{\frac{1}{2}+\delta}}\|_{L^{l^{'}}_tL^{m^{'}}_x} \\
    &\lesssim T^{\beta(p)}\|D^k u\|_{L^l_tL^m_xL^2_{\alpha}}\|u\|^p_{L_t^qL_x^rH_{\alpha}^{\frac{1}{2}+\delta}},
    \endaligned
\end{equation}
with $\beta(p)>0$, where $D$ stands for $\nabla_x,\partial_{\alpha}$ . (We have also used the fractional rule, Lemma A4 in Kato \cite{kato1995nonlinear}.)\vspace{3mm}

Here we choose the indices such that,
\begin{equation}
    \frac{1}{m^{'}}=\frac{1}{m}+\frac{p}{r},\quad \frac{1}{l^{'}}>\frac{1}{l}+\frac{p}{q}.
\end{equation}
It is similar to the NLS case and it is also manageable.\vspace{3mm}

Thus, taking the above estimates into considerations, we can take the proper $T=T(\|u_0\|_{H^2_{x,\alpha}}$ and 
$R=R(\|u_0\|_{H^2_{x,\alpha}}$ such that $\Phi_{u_0}(B_{Z_{T^{'}}}) \subset B_{Z_{T^{'}}}$.\vspace{3mm}

\textbf{Step 2. (To show $\Phi_{u_0}$ is a contraction.)} Let $T,R > 0$ be as in the Step 1. Then there exist $\bar{T}=\bar{T}(\|u_0\|_{
H^2_{x,\alpha}})<T$ such that $\Phi_{u_0}$ is a contraction on $B_{Z_{\bar{T}}}(0,R)$, equipped with the norm $L_{\bar{T}}^qL_x^rL^2_{\alpha}$.\vspace{3mm}

Now we check that the mapping $\Phi_{u_0}$ is a contraction as follows. Using Strichartz estimate and Sobolev embedding,
\begin{equation}
\aligned
    \| \Phi_{u_0}(v_1)-\Phi_{u_0}(v_2)\|_{L_t^qL_x^rL^2_{\alpha}} & \lesssim \| |v_1|^pv_1- |v_2|^pv_2 \|_{L_t^{\tilde{q}^{'}}L_x^{\tilde{r}^{'}}L^2_{\alpha}}\\
    &\lesssim \| \|v_1-v_2\|_{L^2_{\alpha}}(\|v_1\|^p_{L_{\alpha}^{\infty}}+\|v_2\|^p_{L_{\alpha}^{\infty}})\|_{L_t^{\tilde{q}^{'}}L_x^{\tilde{r}^{'}}} \\
    &\lesssim T^{\beta(p)} \| v_1-v_2\|_{L_t^qL_x^rL^2_{\alpha}} \parenthese{ \|v_1\|^{p}_{L_t^qL_x^rH_{\alpha}^{\frac{1}{2}+\delta}}+\|v_2\|^{p}_{L_t^qL_x^rH_{\alpha}^{\frac{1}{2}+\delta}}},
\endaligned    
\end{equation}
with $\beta(p)>0$. And we conclude by taking $\bar{T}$ small sufficiently.\vspace{3mm}

\textbf{Step 3. (Uniqueness and Existence in $Z$)}

It is the same to the NLS case \cite{TV2} so we omit it. It follows from the contraction mapping argument.\vspace{3mm} 

\textbf{Step 4. $u\in \mathcal{C}((-T,T);H^2_{x,{\alpha}})$}

It is the same to the NLS case \cite{TV2} so we omit it. We just use Strichartz estimates again as in Step 1 to guarantee that $u(t,x,y)\in \mathcal{C}((-T,T);H^2_{x,{\alpha}}) $.\vspace{3mm}

\textbf{Step 5. (Unconditional uniqueness)}
We prove that for $u_1, u_2 \in C ((-T,T); H^2_{x,\alpha})$ are fixed points of $\Phi_{u_0}$, then $u_1 = u_2$. 

Considering the difference of the integral equations satisfied by $u_1$ and $u_2$ and using Strichartz estimates,
\begin{equation}
    \aligned
  \|u_1-u_2\|_{L^l_tL^m_xL^2_{\alpha}}&  \lesssim \||u_1|^pu_1-|u_2|^pu_2\|_{L^{l^{'}}_tL^{m^{'}}_xL^2_{\alpha}}  \\
    &\lesssim \|u_1-u_2\|_{L^l_tL^m_xL^2_{\alpha}} \parenthese{\|u_1\|^{p}_{L_t^{\frac{lp}{l-2}}L_x^{\frac{mp}{m-2}}H_{\alpha}^{\frac{1}{2}+\delta}}+\|u_2\|^{p}_{L_t^{\frac{lp}{l-2}}L_x^{\frac{mp}{m-2}}H_{\alpha}^{\frac{1}{2}+\delta}}}\\
    &\lesssim \|u_1-u_2\|_{L^l_tL^m_xL^2_{\alpha}}T^{\beta(p)}\parenthese{\|u_1\|^{p}_{L_t^{\infty}L_x^{\frac{mp}{m-2}}H_{\alpha}^{\frac{1}{2}+\delta}}+\|u_2\|^{p}_{L_t^{\infty}L_x^{\frac{mp}{m-2}}H_{\alpha}^{\frac{1}{2}+\delta}}},
    \endaligned
\end{equation}
with $\beta(p)>0$. \vspace{3mm}

We conclude the proof of uniqueness by selecting $T$
small enough and using the Sobolev inequality  such that 
\begin{equation}
   \|v\|_{L_t^{\infty}L_x^{\frac{mp}{m-2}}H_{\alpha}^{\frac{1}{2}+\delta}} \leq \|v\|_{H^2_{x,y}}.
\end{equation}
We note that $m$ satisfies (see Lemma \ref{index1})
\begin{equation}
    2<\frac{mp}{m-2}<\frac{2d}{d-3}.
\end{equation}

\section{Morawetz estimates}\label{sec Morawetz}
In this section, we derive a Morawetz type estimate for fourth-order Schr\"odinger equations on waveguides $\R^d \times \T$. This estimate is crucial for obtaining the decay property and then the scattering for \eqref{main}. Several ideas have been combined to obtain this estimate, see \cite{MWZ,Pau2,TV2}.

For $x \in \R^d$ and $r \geq 0$, we define $Q^d (x,r)$ to be a $r$ dilation of the unit cube centered at $x$, namely,
\begin{align*}
Q^d (x,r) = x + [-r,r]^d. 
\end{align*}

\begin{proposition}[Morawetz estimates on $\R^d \times \T$]\label{Morawetz}
Let $u (t,x,\alpha) \in C(\R; H_{x,\alpha}^2 (\R^d \times \T))$ be a global solution to \eqref{maineq} with $\frac{8}{d} < p < \frac{8}{d-3}$ and $d \geq 5$.
Then  for every $r >0$, there exists $C = C(r)$ such that
\begin{align*}
\int_{\R} \parenthese{\sup_{x_0 \in \R^d} \iint_{Q^d(x_0 , r) \times \T} \abs{ u(t,x,\alpha)}^2 \, dx d\alpha}^{\frac{p+4}{2}} \, dt  \leq C  \norm{u_0}_{H_{x, \alpha}^2 (\R^d \times \T)}^4 .
\end{align*}
\end{proposition}
\begin{remark}\label{rmk: Mora}
We underline that Proposition \ref{Morawetz} can be extended to the case that the transverse factor is any compact manifold $M^k_y$ ($k$-dimensional) instead of $\mathbb{T}$ and the flat measure $d\alpha$ is replaced by the intrinsic measure $dvol_{M^k_y}$.
\end{remark}
\begin{proof}[Proof of Proposition \ref{Morawetz}]
Let $u,v$ be solutions to the following equations respectively
\begin{align*}
( i \partial_t + \Delta_{x,\alpha}^2 ) u & = F(u),\\
( i \partial_t + \Delta_{y,\beta}^2 ) v & = F(v),
\end{align*}
where $F(u) = - \abs{u}^p u$.

For 
\begin{align*}
z : = (x,\alpha , y ,\beta) \in (\R^d \times \T) \times ( \R^d \times \T) = \{ (x,\alpha , y ,\beta): x , y \in \R^d, \alpha,\beta \in \T \} , 
\end{align*}
define the tensor product for the solutions $u$ and $v$
\begin{align*}
w : = (u  \otimes v) (t,z) = u(t,x,\alpha) v(t,y,\beta) .
\end{align*}
It is easy to verify that   $w=u  \otimes v$  satisfies the following equation
\begin{align*}
( i \partial_t + \Delta^2 ) w = F(u) \otimes v + F(v) \otimes u ,
\end{align*}
where $\Delta^2 =  \Delta_{x,\alpha}^2 + \Delta_{y,\beta}^2$.
In fact,  a direct calculation gives
\begin{align*}
( i \partial_t + \Delta^2 ) w & = ( i \partial_t + \Delta_{x,\alpha}^2 + \Delta_{y,\beta}^2 ) u(t,x,\alpha) v(t,y,\beta) \\
& = i \partial_t (uv) + (\Delta_{x,\alpha}^2 u)v + (\Delta_{y,\beta}^2 v) u\\
& = F(u)v + F(v)u = F(u) \otimes v + F(v) \otimes u .
\end{align*}

Now define the following the Morawetz action $M_a^{\otimes_2} (t) $ in the spirit of \cite{TV2} corresponding to $w=u  \otimes v$ by
\begin{align}\label{eq M}
\begin{aligned}
M_a^{\otimes_2} (t) & : = 2 \int_{(\R^d \times \T) \otimes ( \R^d \times \T)} \nabla_{x,y} a(x,y) \cdot \im [ \overline{w} (z) \nabla_{x,y} w (z) ] \, dz\\
& = 2 \int_{(\R^d \times \T) \otimes ( \R^d \times \T)} \nabla_{x,y} a(x,y) \cdot \im [ \overline{u \otimes v} (z) \nabla_{x,y} (u \otimes v) (z)] \, dz
\end{aligned}
\end{align}
where $\nabla_{x,y} = (\nabla_x, \nabla_y)$. Note that in \eqref{eq M}, $a (x,y)$ is a function independent on the torus direction $ \alpha , \beta$.

Next, we compute the derivative of $M_a^{\otimes_2} (t)$
\begin{align}\label{eq M_t}
\begin{aligned}
\partial_t M_a^{\otimes_2} (t) & = 2 \int_{(\R^d \times \T) \otimes ( \R^d \times \T)} \nabla_{x,y} a(x,y) \cdot \im [ \partial_t\overline{w} (z) \nabla_{x,y} w (z)] \, dz \\
& \quad + 2 \int_{(\R^d \times \T) \otimes ( \R^d \times \T)} \nabla_{x,y} a(x,y) \cdot \im [ \overline{w} (z)  \nabla_{x,y} \partial_t w (z)] \, dz  .
\end{aligned}
\end{align}
Using  the equation of $w$
\begin{align*}
i \partial_t \overline{w} & = \Delta^2 \overline{w} - \overline{F(u) \otimes v} -  \overline{F(v) \otimes u} = \Delta^2 \overline{w} + \abs{u}^p \overline{u} \overline{v} + \abs{v}^p \overline{u} \overline{v} ,\\
- i \partial_t w & =  \Delta^2 w - F(u) \otimes v  -  F(v) \otimes u  =  \Delta^2 w + \abs{u}^p uv + \abs{v}^p uv ,
\end{align*}
we write the imaginary parts in  \eqref{eq M_t}   as
\begin{align}\label{eq im}
\begin{aligned}
\im [\partial_t \overline{w} \partial_j^{x,y} w] & = \re [-i \partial_t \overline{w} \partial_j^{x,y} w] = - \re [ (\Delta^2 \overline{w} + \abs{u}^p \overline{u} \overline{v} + \abs{v}^p \overline{u} \overline{v} )\partial_j^{x,y} w ] ,\\
\im [\overline{w}  \partial_j^{x,y} \partial_t w] & = \re [-i \overline{w}  \partial_j^{x,y} \partial_t w] = \re [\partial_j^{x,y} (\Delta^2 w + \abs{u}^p uv + \abs{v}^p uv) \overline{w}] .
\end{aligned}
\end{align}

Then from \eqref{eq M_t} and \eqref{eq im}, the derivative of $M_a^{\otimes_2} (t)$ becomes
\begin{align*}
\partial_t M_a^{\otimes_2} (t) & = - 2 \int_{(\R^d \times \T) \otimes ( \R^d \times \T)} \partial_j^{x,y} a(x,y) \re [ (\Delta^2 w + \abs{u}^p \overline{u} \overline{v} + \abs{v}^p \overline{u} \overline{v} ) \partial_j^{x,y}  w (z)] \, dz  \\
& \quad +  2 \int_{(\R^d \times \T) \otimes ( \R^d \times \T)}  \partial_j^{x,y} a(x,y) \re [ \partial_j^{x,y} (\Delta^2 w + \abs{u}^p uv+ \abs{v}^p uv)  \overline{w} (z)] \, dz  \\
& = - 2 \int_{(\R^d \times \T) \otimes ( \R^d \times \T)} \partial_j^{x,y} a(x,y) \re [ ( \Delta_{x,y}^2 w + \abs{u}^p \overline{u} \overline{v} + \abs{v}^p \overline{u} \overline{v} ) \partial_j^{x,y}  w (z)] \, dz  \\
& \quad +  2 \int_{(\R^d \times \T) \otimes ( \R^d \times \T)}  \partial_j^{x,y} a(x,y) \re [ \partial_j^{x,y} (\Delta_{x,y}^2 w + \abs{u}^p uv+ \abs{v}^p uv)  \overline{w} (z)] \, dz  \\
& \quad - 2 \int_{(\R^d \times \T) \otimes ( \R^d \times \T)} \partial_j^{x,y} a(x,y) \re [ (\Delta^2 -\Delta_{x,y}^2)  w  \partial_j^{x,y}  w (z)] \, dz  \\
& \quad +  2 \int_{(\R^d \times \T) \otimes ( \R^d \times \T)}  \partial_j^{x,y} a(x,y) \re [ \partial_j^{x,y} (\Delta^2 -\Delta_{x,y}^2) w  \overline{w} (z)] \, dz  \\
& =: M_1 + M_2 + M_3 + M_4.
\end{align*}
Note here when writing $\partial_j   f \partial_j   g$, we mean the Einstein summation convention $\sum_j \partial_j   f \partial_j   g$, and we will be constantly using this convention in the rest of calculation in this section. 
Also notice that  in the last step, we split the bi-Laplacian into the Euclidean directions $\Delta_{x,y}^2$ and the mixed/non-Euclidean directions $\Delta^2 -\Delta_{x,y}^2$. The reason why we separate the Euclidean directions is that we would like to used the available Morawetz estimates in \cite{MWZ} to take care of   $M_1 + M_2$. That is, 
\begin{align*}
M_1 + M_2  & = 2 \re  \int_{\T  \otimes \T}  \int_{\R^d  \otimes \R^d} -\frac{1}{2} (\Delta_x^3 a \abs{u}^2 \abs{v}^2+ \Delta_y^3 a \abs{u}^2 \abs{v}^2) + (\Delta_x^2 a \abs{\nabla_x u}^2 \abs{v}^2 + \Delta_y^2 a \abs{\nabla_y v}^2 \abs{u}^2) \\
& \quad + 2 (\partial_{jk}^x \Delta_x a \,  \partial_j^x \overline{u} \,  \partial_k^x u \abs{v}^2 + \partial_{jk}^y \Delta_y a \,  \partial_j^y \overline{v} \, \partial_k^y v \abs{u}^2) - 4 (\partial_{jk}^x a \, \partial_{ij}^x \overline{u} \, \partial_{ik}^x u \abs{v}^2 +  \partial_{jk}^y a \, \partial_{ij}^y \overline{v} \, \partial_{ik}^y v  \abs{u}^2) \, dz \\
& \quad  - 2   \int_{\T  \otimes \T} \int_{\R^d  \otimes \R^d}  \partial_j^{x,y} a(x,y)   \{ \abs{u}^p uv + \abs{v}^p uv, w \}_p^j \, dz 
\end{align*}
where the Poisson bracket with its derivatives are defined as follows
\begin{align*}
\{ f,g\}_p & = \re  [f \nabla \overline{g} - g \nabla \overline{f}],\\
\{ f,g\}_p^j & = \re  [f \partial_j^{x,y} \overline{g} - g \partial_j^{x,y} \overline{f}].
\end{align*}

Next, we focus on the  terms $M_3$ and $M_4$ with derivatives in torus directions. 

For $M_3$, thanks to the symmetry in $x,y$ and $\alpha,\beta$, we consider the following two cases: $\partial_{x_i x_i \alpha \alpha}$ and $\partial_{\alpha \alpha \alpha \alpha}$.

Preforming integration by parts, we obtain the contribution of $\partial_{x_i x_i \alpha \alpha}$ in $M_3$
\begin{align*}
M_{3a}  & : = - \int_{(\R^d \times \T) \otimes ( \R^d \times \T)}  \partial_j^{x,y} a(x,y) \re [\partial_{x_i x_i \alpha \alpha} \overline{w} \, \partial_j^{x,y} w] \, dz \\
& =  \re \int_{(\R^d \times \T) \otimes ( \R^d \times \T)}  \partial_j^{x,y} a(x,y) \, \partial_{x_i x_i \alpha } \overline{w} \, \partial_j^{x,y} \partial_{  \alpha } w \, dz \\
& = - \re \int_{(\R^d \times \T) \otimes ( \R^d \times \T)}  \partial_{x_i \alpha} \overline{w} \, \partial_{x_i} (\partial_j^{x,y} a(x,y) \, \partial_j^{x,y} \partial_{\alpha} w) \, dz\\
& = - \re \int_{(\R^d \times \T) \otimes ( \R^d \times \T)}  \partial_{x_i \alpha} \overline{w} \, \partial_{x_i} \partial_j^{x,y} a(x,y) \, \partial_j^{x,y} \partial_{\alpha} w \, dz - \re \int_{(\R^d \times \T) \otimes ( \R^d \times \T)}  \partial_{x_i \alpha} \overline{w} \, \partial_j^{x,y} a(x,y) \, \partial_j^{x,y} \partial_{x_i \alpha} w \, dz\\
& = - \re \int_{(\R^d \times \T) \otimes ( \R^d \times \T)}  \partial_{x_i \alpha} \overline{w} \, \partial_{x_i} \partial_j^{x,y} a(x,y) \, \partial_j^{x,y} \partial_{\alpha} w \, dz - \frac{1}{2}  \int_{(\R^d \times \T) \otimes ( \R^d \times \T)}   \partial_j^{x,y} a(x,y) \, \partial_j^{x,y} \abs{ \partial_{x_i \alpha} w }^2 \, dz\\
& =- \re \int_{(\R^d \times \T) \otimes ( \R^d \times \T)}  \partial_{x_i \alpha} \overline{w} \, \partial_{x_i} \partial_j^{x,y} a(x,y) \, \partial_j^{x,y} \partial_{\alpha} w \, dz + \frac{1}{2}  \int_{(\R^d \times \T) \otimes ( \R^d \times \T)}   \partial_{jj}^{x,y} a(x,y)  \abs{ \partial_{x_i \alpha} w }^2 \, dz\\
& =  - \re \int_{(\R^d \times \T) \otimes ( \R^d \times \T)}  \partial_{x_i \alpha} \overline{w} \, \partial_{x_i} \partial_j^{x,y} a(x,y) \, \partial_j^{x,y} \partial_{\alpha} w \, dz + \frac{1}{2}  \int_{(\R^d \times \T) \otimes ( \R^d \times \T)}   \Delta_{x,y}^2 a(x,y)  \abs{ \partial_{x_i \alpha} w }^2 \, dz .
\end{align*}
And the contribution of $\partial_{\alpha \alpha \alpha \alpha}$ in $M_3$ is given by
\begin{align*}
M_{3b}& := - \int_{(\R^d \times \T) \otimes ( \R^d \times \T)}  \partial_j^{x,y} a(x,y) \re [\partial_{\alpha \alpha \alpha \alpha} \overline{w} \, \partial_j^{x,y} w] \, dz\\
& = - \re \int_{(\R^d \times \T) \otimes ( \R^d \times \T)}  \partial_j^{x,y} a(x,y) \, \partial_{\alpha \alpha } \overline{w} \, \partial_j^{x,y} \partial_{\alpha \alpha } w \, dz \\
& = - \frac{1}{2} \int_{(\R^d \times \T) \otimes ( \R^d \times \T)}  \partial_j^{x,y} a(x,y)  \, \partial_j^{x,y}  \abs{\partial_{\alpha \alpha} w}^2 \, dz \\
& = \frac{1}{2} \int_{(\R^d \times \T) \otimes ( \R^d \times \T)}  \partial_{jj}^{x,y} a(x,y)   \abs{\partial_{\alpha \alpha} w}^2 \, dz \\
& =  \frac{1}{2} \int_{(\R^d \times \T) \otimes ( \R^d \times \T)}  \Delta_{x,y}^2 a(x,y)  \abs{ \partial_{\alpha \alpha} w }^2 \, dz .
\end{align*}

Similarly, for $M_4$, we also only consider the following two cases: $\partial_{x_i x_i \alpha \alpha}$ and $\partial_{\alpha \alpha \alpha \alpha}$. The contribution of $\partial_{x_i x_i \alpha \alpha}$ to $M_4$ can be written as
\begin{align*}
M_{4a} & := \int_{(\R^d \times \T) \otimes ( \R^d \times \T)}  \partial_j^{x,y} a(x,y) \re [ \partial_j^{x,y} \partial_{x_i x_i \alpha \alpha} w \, \overline{w}] \, dz \\
& = - \re \int_{(\R^d \times \T) \otimes ( \R^d \times \T)}  \partial_j^{x,y} a(x,y) \, \partial_j^{x,y} \partial_{x_i x_i  \alpha} w \, \partial_{\alpha} \overline{w} \, dz \\
& = \re \int_{(\R^d \times \T) \otimes ( \R^d \times \T)}  \partial_j^{x,y} \partial_{ x_i  \alpha} w \, \partial_{x_i} (\partial_j^{x,y} a(x,y) \,  \partial_{\alpha} \overline{w} ) \, dz \\
& = \re \int_{(\R^d \times \T) \otimes ( \R^d \times \T)}  \partial_j^{x,y} \partial_{ x_i  \alpha} w \, \partial_j^{x,y} \partial_{x_i}  a(x,y)  \, \partial_{\alpha} \overline{w}  \, dz +  \re \int_{(\R^d \times \T) \otimes ( \R^d \times \T)}  \partial_j^{x,y} \partial_{ x_i  \alpha} w \, \partial_j^{x,y} a(x,y) \,  \partial_{x_i \alpha} \overline{w}  \, dz \\
& = \re \int_{(\R^d \times \T) \otimes ( \R^d \times \T)}  \partial_j^{x,y} \partial_{ x_i  \alpha} w \, \partial_j^{x,y} \partial_{x_i}  a(x,y) \,  \partial_{\alpha} \overline{w}  \, dz  +  \frac{1}{2} \int_{(\R^d \times \T) \otimes ( \R^d \times \T)}  \partial_j^{x,y} a(x,y) \, \partial_j^{x,y}  \abs{\partial_{x_i \alpha} w}^2 \, dz \\
& = \re \int_{(\R^d \times \T) \otimes ( \R^d \times \T)}  \partial_j^{x,y} \partial_{ x_i  \alpha} w \, \partial_j^{x,y} \partial_{x_i}  a(x,y) \,  \partial_{\alpha} \overline{w}  \, dz  - \frac{1}{2} \int_{(\R^d \times \T) \otimes ( \R^d \times \T)}  \Delta_{x,y}^2 a(x,y) \abs{\partial_{x_i \alpha} w}^2 \, dz ,
\end{align*}
while the contribution of  $\partial_{\alpha \alpha \alpha \alpha}$ to $M_4$ is 
\begin{align*}
M_{4b} & := \int_{(\R^d \times \T) \otimes ( \R^d \times \T)}  \partial_j^{x,y} a(x,y) \re [\partial_j^{x,y} \partial_{\alpha \alpha \alpha \alpha} w \, \overline{w}] \, dz\\
& = \re \int_{(\R^d \times \T) \otimes ( \R^d \times \T)}  \partial_j^{x,y} a(x,y) \, \partial_j^{x,y} \partial_{\alpha \alpha} w \, \partial_{\alpha \alpha} \overline{w} \, dz\\
& = \frac{1}{2} \int_{(\R^d \times \T) \otimes ( \R^d \times \T)}    \partial_j^{x,y} a(x,y)   \, \partial_j^{x,y} \abs{\partial_{\alpha \alpha} w}^2 \, dz \\
& = - \frac{1}{2} \int_{(\R^d \times \T) \otimes ( \R^d \times \T)}  \partial_{jj}^{x,y} a(x,y)   \abs{\partial_{\alpha \alpha} w}^2 \, dz \\
& = - \frac{1}{2} \int_{(\R^d \times \T) \otimes ( \R^d \times \T)}  \Delta_{x,y}^2 a(x,y)  \abs{ \partial_{\alpha \alpha} w }^2 \, dz .
\end{align*}

Now gathering the computation above and integrating by parts, we have
\begin{align*}
& \quad M_3 + M_4= M_{3a} + M_{3b} + M_{4a} +M_{4b} \\
& = - \re \int_{(\R^d \times \T) \otimes ( \R^d \times \T)}  \partial_{x_i \alpha} \overline{w} \, \partial_{x_i} \partial_j^{x,y} a(x,y) \, \partial_j^{x,y} \partial_{\alpha} w \, dz + \re \int_{(\R^d \times \T) \otimes ( \R^d \times \T)}  \partial_j^{x,y} \partial_{ x_i  \alpha} w \, \partial_j^{x,y} \partial_{x_i}  a(x,y) \,  \partial_{\alpha} \overline{w}  \, dz  \\
& = - \re \int_{(\R^d \times \T) \otimes ( \R^d \times \T)}  \partial_{x_i \alpha} \overline{w} \, \partial_{x_i} \partial_j^{x,y} a(x,y) \, \partial_j^{x,y} \partial_{\alpha} w \, dz  - \re \int_{(\R^d \times \T) \otimes ( \R^d \times \T)}  \partial_{x_i} (\partial_j^{x,y} \partial_{x_i}  a(x,y) \,  \partial_{\alpha} \overline{w} ) \, \partial_j^{x,y} \partial_{ \alpha} w   \, dz  \\
& = - 2 \re \int_{(\R^d \times \T) \otimes ( \R^d \times \T)}  \partial_{x_i \alpha} \overline{w} \, \partial_{x_i} \partial_j^{x,y} a(x,y) \, \partial_j^{x,y} \partial_{\alpha} w \, dz  - \re \int_{(\R^d \times \T) \otimes ( \R^d \times \T)}  \partial_j^{x,y} \partial_{x_i x_i}  a(x,y)  \, \partial_{\alpha} \overline{w} \,  \partial_j^{x,y} \partial_{ \alpha} w   \, dz  \\
& =: I + II.
\end{align*}
Another integration by parts on $II$ gives
\begin{align*}
II & = - \re \int_{(\R^d \times \T) \otimes ( \R^d \times \T)}  \partial_j^{x,y} \partial_{x_i x_i}  a(x,y) \,  \partial_{\alpha} \overline{w} \,  \partial_j^{x,y} \partial_{ \alpha} w   \, dz  \\
& = \re \int_{(\R^d \times \T) \otimes ( \R^d \times \T)}  \partial_{jj}^{x,y} \partial_{x_i x_i}  a(x,y) \,  \partial_{\alpha} \overline{w} \,   \partial_{ \alpha} w   \, dz  + \re \int_{(\R^d \times \T) \otimes ( \R^d \times \T)}  \partial_j^{x,y} \partial_{x_i x_i}  a(x,y) \,  \partial_{\alpha} w \,  \partial_j^{x,y} \partial_{ \alpha} \overline{w}   \, dz \\
& = \re \int_{(\R^d \times \T) \otimes ( \R^d \times \T)}  \partial_{jj}^{x,y} \partial_{x_i x_i}  a(x,y) \,  \partial_{\alpha} \overline{w} \,   \partial_{ \alpha} w   \, dz  - II ,
\end{align*}
which yields
\begin{align*}
II = \frac{1}{2} \re \int_{(\R^d \times \T) \otimes ( \R^d \times \T)}  \partial_{jj}^{x,y} \partial_{x_i x_i}  a(x,y) \,  \partial_{\alpha} \overline{w} \,   \partial_{ \alpha} w   \, dz =   \int_{(\R^d \times \T) \otimes ( \R^d \times \T)}  \Delta_x^2   a(x,y)  \abs{\partial_{\alpha} w}^2 +  \Delta_y^2   a(x,y)  \abs{\partial_{\beta} w}^2   \, dz .
\end{align*}

Splitting  $\partial_j^{x,y}$ into $\partial_j^{x}$ and $\partial_j^{y}$ and integrating by parts, we obtain
\begin{align*}
I & = - 2 \re \int_{(\R^d \times \T) \otimes ( \R^d \times \T)}  \partial_{x_i \alpha} \overline{w} \, \partial_{x_i} \partial_j^{x,y} a(x,y) \, \partial_j^{x,y} \partial_{\alpha} w \, dz \\
& =  - 2 \re \int_{(\R^d \times \T) \otimes ( \R^d \times \T)}  \partial_{x_i \alpha} \overline{w} \, \partial_{x_i x_j} a(x,y) \, \partial_{x_j \alpha} w \, dz   - 2 \re \int_{(\R^d \times \T) \otimes ( \R^d \times \T)}  \partial_{x_i \alpha} \overline{w} \, \partial_{x_i y_j} a(x,y) \, \partial_{y_j \alpha} w \, dz \\
& = - 2 \re \int_{(\R^d \times \T) \otimes ( \R^d \times \T)}  \partial_{x_i \alpha} \overline{u} \, \partial_{x_j \alpha} u \, \partial_{x_i x_j} a(x,y)  \abs{v}^2 \, dz - 2 \re \int_{(\R^d \times \T) \otimes ( \R^d \times \T)}  \partial_{x_i \alpha} \overline{u} \overline{v} \, \partial_{x_i x_j} a(x,y) \, \partial_{y_j } v \, \partial_{\alpha} u \, dz \\
& = - 2 \re \int_{(\R^d \times \T) \otimes ( \R^d \times \T)}  \partial_{x_i \alpha} \overline{u} \, \partial_{x_j \alpha} u \, \partial_{x_i x_j} a(x,y)  \abs{v}^2 \, dz - \frac{1}{2} \int_{(\R^d \times \T) \otimes ( \R^d \times \T)}  \partial_{x_i } \abs{\partial_{\alpha} u}^2 \,   \partial_{x_i x_j} a(x,y) \, \partial_{y_j } \abs{v}^2  \, dz \\
& = - 2 \re \int_{(\R^d \times \T) \otimes ( \R^d \times \T)}  \partial_{x_i \alpha} \overline{u} \, \partial_{x_j \alpha} u \, \partial_{x_i x_j} a(x,y)  \abs{v}^2 \, dz + \frac{1}{2} \int_{(\R^d \times \T) \otimes ( \R^d \times \T)}   \abs{\partial_{\alpha} u}^2   \partial_{x_i x_i x_j x_j} a(x,y) \abs{v}^2  \, dz \\
& = - 2 \re \int_{(\R^d \times \T) \otimes ( \R^d \times \T)}  \partial_{x_i \alpha} \overline{u} \, \partial_{x_j \alpha} u \, \partial_{x_i x_j} a(x,y)  \abs{v}^2 \, dz + \frac{1}{2} \int_{(\R^d \times \T) \otimes ( \R^d \times \T)}     \Delta_x^2 a(x,y)  \abs{\partial_{\alpha} w}^2 +   \Delta_y^2 a(x,y)  \abs{\partial_{\beta} w}^2  \, dz 
\end{align*}

Therefore,
\begin{align*}
& \quad M_3 + M_4 = I + II\\
& = - 2 \re \int_{(\R^d \times \T) \otimes ( \R^d \times \T)}  \partial_{x_i \alpha} \overline{u} \, \partial_{x_j \alpha} u \, \partial_{ij}^x a(x,y)  \abs{v}^2 \, dz - 2 \re \int_{(\R^d \times \T) \otimes ( \R^d \times \T)}  \partial_{y_i \beta} \overline{v} \, \partial_{y_j \beta} v \, \partial_{ij}^y a(x,y)  \abs{u}^2 \, dz \\
& \quad +  \frac{3}{2} \int_{(\R^d \times \T) \otimes ( \R^d \times \T)}     \Delta_x^2 a(x,y)  \abs{\partial_{\alpha} w}^2  \, dz +  \frac{3}{2} \int_{(\R^d \times \T) \otimes ( \R^d \times \T)}     \Delta_y^2 a(x,y)  \abs{\partial_{\beta} w}^2  \, dz  . 
\end{align*}

Finally putting all the computation together, we obtain the  derivative of $M_a^{\otimes_2} (t)$  in the following form
\begin{align}
& \quad \partial_t M_a^{\otimes_2} (t)  = M_1 + M_2 + M_3 + M_4 \label{eq M_t full}\\
& =  2 \re  \int_{\T  \otimes \T}  \int_{\R^d  \otimes \R^d} -\frac{1}{2} (\Delta_x^3 a \abs{u}^2 \abs{v}^2+ \Delta_y^3 a \abs{u}^2 \abs{v}^2) + (\Delta_x^2 a \abs{\nabla_x u}^2 \abs{v}^2 + \Delta_y^2 a \abs{\nabla_y v}^2 \abs{u}^2) \label{eq M1}\\
&  + 2 (\partial_{jk}^x \Delta_x a \, \partial_j^x \overline{u} \, \partial_k^x u \abs{v}^2 + \partial_{jk}^y \Delta_y a \, \partial_j^y \overline{v} \, \partial_k^y v \abs{u}^2) - 4 (\partial_{jk}^x a \, \partial_{ij}^x \overline{u} \, \partial_{ik}^x u \abs{v}^2 +  \partial_{jk}^y a \, \partial_{ij}^y \overline{v} \, \partial_{ik}^y v  \abs{u}^2) \, dz \label{eq M2}\\
&  - 2   \int_{\T  \otimes \T} \int_{\R^d  \otimes \R^d}  \partial_j^{x,y} a(x,y)   \{ \abs{u}^p uv + \abs{v}^p uv, w \}_p^j \, dz  \label{eq M3}\\
&   - 2 \re \int_{(\R^d \times \T) \otimes ( \R^d \times \T)}  \partial_{x_i \alpha} \overline{u} \, \partial_{x_j \alpha} u \, \partial_{ij}^x a(x,y)  \abs{v}^2 \, dz - 2 \re \int_{(\R^d \times \T) \otimes ( \R^d \times \T)}  \partial_{y_i \beta} \overline{v} \, \partial_{y_j \beta} v \, \partial_{ij}^y a(x,y)  \abs{u}^2 \, dz  \label{eq M4}\\
&  +  \frac{3}{2} \int_{(\R^d \times \T) \otimes ( \R^d \times \T)}     \Delta_x^2 a(x,y)  \abs{\partial_{\alpha} w}^2  \, dz +  \frac{3}{2} \int_{(\R^d \times \T) \otimes ( \R^d \times \T)}     \Delta_y^2 a(x,y)  \abs{\partial_{\beta} w}^2  \, dz   \label{eq M5}.
\end{align}

After finishing the  derivative of Morawetz action, we choose   $a (x,y ) = \inner{x-y} : = (1 + \abs{x-y})^{\frac{1}{2}}$ in \eqref{eq M_t full}.
Now we are   one step away from a Morawetz type inequality, that is, determining the signs of every single term in \eqref{eq M_t full}. In order to hold the inequality, they have to be with right signs, and our goal is to keep only the first term in \eqref{eq M1} and \eqref{eq M3} and drop all other terms.

A direct calculation gives us the following derivatives of the chosen function $a(x,y)$:
\begin{align}\label{eq list}
\begin{aligned}
& \Delta_x a  = \Delta_y a= \frac{d-1} { \inner{x-y}} + \frac{1}{\inner{x-y}^3},\\
&  \partial_{ij}^x a = \partial_{ij}^y a = \frac{\delta_{ij}}{\inner{x-y}} - \frac{(x-y)_i (x-y)_j }{\inner{x-y}^{3}},\\
& \Delta_x^2 a = \Delta_y^2 a  = - \frac{(d-1) (d-3) }{\inner{x-y}^{3}} - \frac{6 (d-3) }{\inner{x-y}^{5}} - \frac{15}{ \inner{x-y}^{7}},\\
& \Delta_x^3 a = \Delta_y^3 a  = \frac{3 (d-1) (d-3) (d-5) }{\inner{x-y}^{5}} + \frac{45 (d-3) (d-5)}{\inner{x-y}^{7}} + \frac{315 (d-5)}{\inner{x-y}^{9}} + \frac{945 }{\inner{x-y}^{11}},\\
& \partial_{ij}^x \Delta_x a = \partial_{ij}^y \Delta_y a  = -\frac{(d-1) \delta_{ij} }{\inner{x-y}^{3}} + \frac{3 (d-3) (x-y)_i (x-y)_j}{\inner{x-y}^{5}} - \frac{3 \delta_{ij} }{\inner{x-y}^{5}} + \frac{15 (x-y)_i (x-y)_j }{\inner{x-y}^{7}}.
\end{aligned}
\end{align}

For any vector $e \in \R^d$, and a function $u$, we define
\begin{align*}
& \nabla_e u  = (e \cdot \nabla u) \frac{e}{\abs{e}^2} ,\\
& \nabla_e^{\perp} u = \nabla u - \nabla_e u .
\end{align*}
In the following calculation, we set $e = x-y$.

First we consider the term $M_1 + M_2$, which  from the contribution of Euclidean bi-Laplacian $\Delta_{x,y}^2$. For $d \geq 3$, we have the bounds for the second term in \eqref{eq M1}
\begin{align*}
\Delta_x^2 a \abs{\nabla_x u}^2 & = \parenthese{- \frac{(d-1) (d-3) }{\inner{x-y}^{3}} - \frac{6 (d-3) }{\inner{x-y}^{5}} - \frac{15}{ \inner{x-y}^{7}} } \abs{\nabla_x u}^2 \\
& = \parenthese{- \frac{(d+5) (d-3)}{\inner{x-y}^{5}} - \frac{ (d-1) (d-3)\abs{x-y}^2}{ \inner{x-y}^{5}} -\frac{15 }{\inner{x-y}^{7}}} \abs{\nabla_x u}^2 \\
& \leq - \frac{(d+5) (d-3) \abs{\nabla_e u}^2 }{\inner{x-y}^{5}} ,
\end{align*}
and for the first term in \eqref{eq M2}
\begin{align*}
\partial_{jk}^x \Delta_x a \, \partial_j^x \overline{u} \, \partial_k^x u & = \parenthese{-\frac{(d-1) \delta_{ij} }{\inner{x-y}^{3}} + \frac{3 (d-3) (x-y)_i (x-y)_j}{\inner{x-y}^{5}} - \frac{3 \delta_{ij} }{\inner{x-y}^{5}} + \frac{15 (x-y)_i (x-y)_j }{\inner{x-y}^{7}}}  \partial_j^x \overline{u} \, \partial_k^x u \\
& \leq - \frac{(d-1)\abs{\nabla_x u}^2 }{\inner{x-y}^{3}}  +  \frac{3(d-1)\abs{x-y}^2\abs{\nabla_e u}^2}{\inner{x-y}^5}  - \frac{3\abs{\nabla_x u}^2}{\inner{x-y}^5} +  \frac{15 \abs{x-y}^2 \abs{\nabla_e u}^2}{\inner{x-y}^7} + \frac{12 \abs{\nabla_e u}^2}{\inner{x-y}^5}\\
& \leq \frac{2(d-1)}{\inner{x-y}^3} - \frac{3(d+3) \abs{\nabla_e u}^2}{\inner{x-y}^5} , 
\end{align*}
and also for the second term in \eqref{eq M2}
\begin{align*}
- \partial_{jk}^x \Delta_x a \, \partial_{ij}^x \overline{u} \, \partial_{ik}^x u & = - \parenthese{\frac{\delta_{ij}}{\inner{x-y}} - \frac{(x-y)_i (x-y)_j }{\inner{x-y}^{3}}} \partial_{ij}^x \overline{u} \, \partial_{ik}^x u \\
& = - \frac{\abs{\nabla_x \partial_i^x u}}{\inner{x-y}} + \frac{\abs{x-y}^2 \abs{\nabla_e \partial_i^x u}^2}{\inner{x-y}^3}  \\
& \leq - \frac{\abs{\nabla_x \partial_i^x u}^2 - \abs{\nabla_e \partial_i^x u}^2}{\inner{x-y}} \\
& \leq  - \frac{(d-1)\abs{\nabla_e u}^2}{\abs{x-y}^2 \inner{x-y}} \leq -\frac{(d-1)\abs{\nabla_e u}^2}{\inner{x-y}^3} .
\end{align*}
In the last line above, we used the following inequality as shown in Levandosky-Strauss  \cite{LS} (it is also used in \cite{Pau2}) 
\begin{align*}
\sum_i (\abs{\nabla_x \partial_i^x u}^2 - \abs{\nabla_e \partial_i^x u}^2) \geq \frac{d-1}{\abs{x-y}^2} \abs{\nabla_e u}^2 .
\end{align*}

Collecting all the calculation above, we know that  the integrand in the `to be dropped' term in \eqref{eq M1}, \eqref{eq M2} is bounded by 
\begin{align*}
\parenthese{\Delta_x^2 a \abs{\nabla_x u}^2 + 2 \partial_{jk}^x \Delta_x a \, \partial_j^x \overline{u} \, \partial_k^x u  - 4 \partial_{jk}^x a \, \partial_{ij}^x \overline{u} \, \partial_{ik}^x u } \abs{v}^2 & \leq \square{-6 (d+3) - (d+5)(d-3)} \frac{\abs{\nabla_e u }^2 }{\inner{x-y}^5} \abs{v}^2 \\
& = (-d^2 -8d +45) \frac{\abs{\nabla_e u }^2 }{\inner{x-y}^5} \abs{v}^2 \leq 0 .
\end{align*}
Here notice that in the last inequality we used $-d^2 -8d +45 \leq 0$ when $d \geq 4$.

Therefore, we require $d \geq 4$ to have the right sign, hence
\begin{align*}
& \quad M_1 + M_2  \\
& \leq 2 \re  \int_{\T  \otimes \T}  \int_{\R^d  \otimes \R^d} -\frac{1}{2} (\Delta_x^3 a \abs{u}^2 \abs{v}^2+ \Delta_y^3 a \abs{u}^2 \abs{v}^2)  \, dz - 2   \int_{\T  \otimes \T} \int_{\R^d  \otimes \R^d}  \partial_j^{x,y} a(x,y)   \{ \abs{u}^p uv + \abs{v}^p uv, w \}_p^j \, dz  .
\end{align*}

Next, we turn to the term $M_3 + M_4$. Because of the symmetry in $x,y$ and $\alpha , \beta$, we only consider the first term in \eqref{eq M4} and the first term in \eqref{eq M5}. Using the list \eqref{eq list}, we write the first term in \eqref{eq M4} as 
\begin{align*}
- \partial_{x_i \alpha} \overline{u} \, \partial_{x_j \alpha} u \, \partial_{ij}^x a(x,y) & = -  \partial_{x_i \alpha} \overline{u} \, \partial_{x_j \alpha} u \parenthese{\frac{\delta_{ij}}{\inner{x-y}} - \frac{(x-y)_i (x-y)_j }{\inner{x-y}^{3}} } \\
& = - \frac{\abs{\nabla_x \partial_{\alpha} u}^2}{\inner{x-y}} + \frac{\abs{x-y}^2 \abs{\nabla_e \partial_{\alpha} u}^2}{\inner{x-y}^3} \\
& = - \frac{\abs{\nabla_x \partial_{\alpha} u}^2}{\inner{x-y}^3} - \frac{\abs{x-y}^2 \abs{\nabla_e^{\perp} \partial_{\alpha} u}^2}{\inner{x-y}^3} \leq 0 ,
\end{align*}
and the first term in \eqref{eq M5}  as
\begin{align*}
\Delta_x^2 a \abs{\partial_{\alpha} w}^2 = \parenthese{- \frac{(d-1) (d-3) }{\inner{x-y}^{3}} - \frac{6 (d-3) }{\inner{x-y}^{5}} - \frac{15}{ \inner{x-y}^{7}}}  \abs{\partial_{\alpha} w}^2 \leq 0 . 
\end{align*}

Thanks to the right signs in \eqref{eq M4} and \eqref{eq M5}, we obtain
\begin{align*}
M_3 + M_4 \leq 0 .
\end{align*}
Then,
\begin{align*}
\partial_t M_a^{\otimes_2} (t) & = M_1 + M_2 + M_3 + M_4 \\
& \leq 2 \re  \int_{\T  \otimes \T}  \int_{\R^d  \otimes \R^d} -\frac{1}{2} (\Delta_x^3 a \abs{u}^2 \abs{v}^2+ \Delta_y^3 a \abs{u}^2 \abs{v}^2)  \, dz \\
& \quad - 2   \int_{\T  \otimes \T} \int_{\R^d  \otimes \R^d}  \partial_j^{x,y} a(x,y)   \{ \abs{u}^p uv + \abs{v}^p uv, w \}_p^j \, dz  .
\end{align*}
Noticing that when $d \geq 5$, $\Delta_x^3 a = \Delta_y^3 a \geq 0$, we have
\begin{align}\label{eq MM}
 \int_0^T   \int_{\T  \otimes \T} \int_{\R^d  \otimes \R^d}  \partial_j^{x,y} a(x,y)   \{ \abs{u}^p uv + \abs{v}^p uv, w \}_p^j \, dz dt  \lesssim \sup_{t \in [0,T]} \abs{M_a^{\otimes_2} (t)}. 
\end{align}
Preforming interrogation by parts, we write
\begin{align}\label{eq MM1}
& \quad  \int_0^T   \int_{\T  \otimes \T} \int_{\R^d  \otimes \R^d}  \partial_j^{x,y} a(x,y)   \{ \abs{u}^p uv + \abs{v}^p uv, w \}_p^j \, dz dt\\
& = (1-\frac{2}{p+2}) \int_0^T   \int_{\T  \otimes \T} \int_{\R^d  \otimes \R^d} \Delta_x^2 a \abs{u}^{p+2} \abs{v}^2 + \Delta_y^2 a \abs{v}^{p+2} \abs{u}^2 \, dz dt.
\end{align}
In fact, due to the symmetry in $x$ and $y$, we only consider the following terms in the Poisson bracket. Noticing $\partial_j^{x} (\abs{u}^2) =  u \, \partial_j^{x} \overline{u}+ \overline{u}  \, \partial_j^{x} u  $ and $\abs{u}^p \partial_j^{x} (\abs{u}^2)  = \frac{2}{p+2} \partial_j^{x} (\abs{u}^{p+2})$, and using integration  by parts, we write
\begin{align*}
& \quad   \re  \int_0^T   \int_{\T  \otimes \T} \int_{\R^d  \otimes \R^d}  \partial_j^{x} a(x,y)    [\abs{u}^p uv \,  \partial_j^{x} ( \overline{u v} ) -   u v \, \partial_j^{x}  (\abs{u}^p \overline{u v})] \, dz dt \\
& =  \re  \int_0^T   \int_{\T  \otimes \T} \int_{\R^d  \otimes \R^d} \partial_j^{x} a(x,y)     \abs{u}^p  \abs{v}^2  u \, \partial_j^{x}  \overline{u }    + \partial_j^{x} a(x,y) \abs{u}^p  \abs{v}^2  \overline{u } \,  \partial_j^{x} u  + \Delta_x a(x,y) \abs{u}^{p+2} \abs{v}^2 \, dzdt \\
& =   \int_0^T   \int_{\T  \otimes \T} \int_{\R^d  \otimes \R^d}   \partial_{j}^{x} a(x,y) \abs{u}^{p} \abs{v}^2 \partial_j^x (\abs{u}^2) +   \Delta_x a(x,y)  \abs{u}^{p+2} \abs{v}^2 \, dzdt \\
& =    \int_0^T   \int_{\T  \otimes \T} \int_{\R^d  \otimes \R^d}  \frac{2}{p+2} \partial_{j}^{x} a(x,y) \abs{v}^2 \partial_j^x ( \abs{u}^{p+2}) +   \Delta_x a(x,y)  \abs{u}^{p+2} \abs{v}^2 \, dzdt \\
& =    \int_0^T   \int_{\T  \otimes \T} \int_{\R^d  \otimes \R^d} - \frac{2}{p+2} \Delta_x a(x,y)  \abs{u}^{p+2} \abs{v}^2+  \Delta_x a(x,y)  \abs{u}^{p+2} \abs{v}^2 \, dzdt \\
& =  \int_0^T   \int_{\T  \otimes \T} \int_{\R^d  \otimes \R^d} (1- \frac{2}{p+2}) \Delta_x a(x,y)  \abs{u}^{p+2} \abs{v}^2  \, dzdt ,
\end{align*}
and
\begin{align*}
\re  \int_0^T   \int_{\T  \otimes \T} \int_{\R^d  \otimes \R^d}  \partial_j^{x} a(x,y)    ( \abs{v}^p uv \,  \partial_j^{x} ( \overline{u v} ) -   u v \, \partial_j^{x}  (\abs{v}^p \overline{u v})) \, dz dt  =0 .
\end{align*}

Therefore \eqref{eq MM} with  \eqref{eq MM1}  implies
\begin{align*}
\int_0^T   \int_{\T  \otimes \T} \int_{\R^d  \otimes \R^d} \Delta_x^2 a \abs{u}^{p+2} \abs{v}^2 + \Delta_y^2 a \abs{v}^{p+2} \abs{u}^2 \, dz dt \lesssim \sup_{t \in [0,T]} \abs{M_a^{\otimes_2} (t)}  . 
\end{align*}
Taking $u=v$ in the Morawetz action \eqref{eq M} and combining   the upper bound for $M_a^{\otimes_2} (t)$
\begin{align*}
\sup_{t \in [0,T]} \abs{M_a^{\otimes_2} (t)} \lesssim \norm{u_0}_{H_{x, \alpha}^1 (\R^d \times \T)}^4 \lesssim \norm{u_0}_{H_{x, \alpha}^2 (\R^d \times \T)}^4 , 
\end{align*}
we obtain finally a Morawetz estimate that we desire
\begin{align*}
\int_0^T   \int_{\T  \otimes \T} \int_{\R^d  \otimes \R^d} \Delta_x^2 a \abs{u}^{p+2} \abs{v}^2 + \Delta_y^2 a \abs{v}^{p+2} \abs{u}^2 \, dz dt  \lesssim \norm{u_0}_{H_{x, \alpha}^2 (\R^d \times \T)}^4 .
\end{align*}

However, we prefer to not have the function $a$ appearing in the estimate.  Recall that  $Q^d (x,r)$ is a $r$ dilation of the unit cube centered at $x$,
\begin{align*}
Q^d (x,r) = x + [-r,r]^d. 
\end{align*}
Then we know  $\inf_{Q^d (0, 2r)} \Delta_x (\inner{x}) >0$, and 
\begin{align*}
\int_{\R} \sup_{x_0 \in \R^d} \parenthese{ \iint_{(Q^d (x_0 , r))^2 \otimes \T^2} \abs{u(t, x, \alpha)}^{p+2} \abs{u (t, y,\beta)}^2  \, dz } \, dt \lesssim \norm{u_0}_{H_{x, \alpha}^2 (\R^d \times \T)}^4 .
\end{align*}
By H\"older inequality, we write 
\begin{align*}
\parenthese{ \int_{Q^d (x_0,r) \times \T} \abs{u(t, x, \alpha)}^2 \, dx d \alpha}^{\frac{p+2}{2}} \leq C_r  \int_{Q^d (x_0, r) \times \T} \abs{u (t, x, \alpha)}^{p+2} \, dx d \alpha .
\end{align*}
Finally, we obtain a satisfactory  Morawetz type estimate for the fourth-order Schr\"odinger equation on waveguides 
\begin{align*}
\int_{\R} \parenthese{\sup_{x_0 \in \R^d} \iint_{Q^d(x_0 , r) \times \T} \abs{ u(t,x,\alpha)}^2 \, dx d\alpha}^{\frac{p+4}{2}} \, dt  \leq C_r \norm{u_0}_{H_{x, \alpha}^2(\R^d \times \T)}^4 .
\end{align*}
\end{proof}

\section{Decay property and Scattering}\label{sec Scattering}
In this section, we apply the interaction Morawetz estimate in Proposition \ref{Morawetz} to prove the scattering for \eqref{maineq}. There are three steps. First, we show the decay property for \eqref{maineq} based on the interaction Morawetz estimate (together with Gagliardo–Nirenberg inequality in the waveguide setting) via a contradiction argument. Second, we use the decay property to obtain a spacetime bound for \eqref{maineq}. At last, we apply the spacetime bound to show the scattering behavior. This strategy has the same spirit as the Euclidean case, i.e. proving scattering for energy subcritical NLS in the energy space. See for \cite{cazenave2003semilinear,Taobook}.
\subsection{Step 1: Proof of the decay property}
Based on the Morawetz bound, we aim to show,
\begin{equation}\label{decay}
 \lim_{t \rightarrow \infty} \|u(t,x,\alpha)\|_{L^q_{x,\alpha} (\R^d \times \T)}=0,   
\end{equation}
where $2<q \leq 2+\frac{4}{d+1}$. The Morawetz estimate reads: (for every $r>0$)
\begin{equation}
    \int_{\mathbb{R}} \parenthese{\sup_{x_0\in \mathbb{R}^d} \iint_{Q^d(x_0,r)\times \mathbb{T}^n} |u(t,x,\alpha)|^2 \, dxd\alpha}^{\frac{p+4}{2}} \, dt \lesssim \|u_0\|^4_{H^{2}_{x,\alpha} (\R^d \times \T)}.
\end{equation}
We recall the localized Gagliardo–Nirenberg inequality (see eq. (A-5) of Terracini-Tzvetkov-Visciglia \cite{terracini2014nonlinear}):
\begin{equation}
    \|v\|_{L_{x,\alpha}^{2+\frac{4}{d+n}}  (\R^d \times \T^n)}\lesssim \sup_{x\in \mathbb{R}^d} \parenthese{\|v\|^{2}_{L^2_{Q^{d}_{x}\times \mathbb{T}^n}}}^{\frac{2}{d+n+2}}\|v\|^{\frac{d+n}{d+n+2}}_{H^1(\mathbb{R}^d\times \mathbb{T}^n)}
\end{equation}
where $Q^{d}_{x}=x+[0,1]^d$ for all $x\in \mathbb{R}^d$.

Now we prove the decay property via a contradiction argument. In fact it is sufficient to show that
\begin{equation}
    \lim_{t\rightarrow \infty}\|u(t,x,\alpha)\|_{L_{x,\alpha}^{2+\frac{4}{d+1}} (\R^d \times \T)} =0,
\end{equation}
since other bounds can be obtained by interpolating with the conservation law.

Next, assume that the decay estimate does not hold. Then we deduce  to the existence of a sequence $(t_n, x_n) \in \mathbb{R}\times \mathbb{R}^d$ with $t_n \rightarrow \infty$ and $\varepsilon_0>0$
such that
\begin{equation}
    \inf_{n} \|u(t_n,x,\alpha)\|_{L^2_{Q^{d}(x_n,1)\times \mathbb{T}}} =\varepsilon_0.
\end{equation}
We get the existence of $T > 0$ such that
\begin{equation}
    \inf_{n} \parenthese{ \inf_{t\in (t_n,t_n+T)}\|u(t,x,\alpha)\|_{L^2_{Q^{d}(x_n,2)\times \mathbb{T}}} }\geq \frac{\varepsilon_0}{2}.
\end{equation}
Notice that since $t_n \rightarrow \infty$ as $n \rightarrow \infty$ then we can assume (up to a subsequence) that the
intervals $(t_n,t_n+T)$ are disjoint. In particular we have
\begin{equation}
\aligned
\sum_{n} T(\frac{\varepsilon_0}{2})^{p+4}&\leq \sum_{n}\int_{t_n}^{t_n+T} \parenthese{ \iint_{Q^d(x_n,2)\times \mathbb{T}} |u(t,x,\alpha)|^2 \, dxdy}^{\frac{p+4}{2}} \, dt\\
&\leq \int_{\mathbb{R}}\parenthese{\sup_{z\in \mathbb{R}^d} \iint_{Q^d(z,2)\times \mathbb{T}} |u(t,x,\alpha)|^2 \, dxd\alpha}^{\frac{p+4}{2}} \, dt.
\endaligned
\end{equation}
Thus we get a contradiction since the left hand side is divergent and the right hand side is bounded by the Morawetz estimate, i.e. Proposition \ref{Morawetz}.
\subsection{Step 2: Proof of the spacetime bound}
We aim to show,
\begin{equation}\label{est1}
  u \in L_t^{q_{\theta}}L_x^{r_{\theta}}H_{\alpha}^{\frac{1}{2}+\delta}(\mathbb{R}_t\times \R^d \times \T)  
\end{equation}
and
\begin{equation}\label{est2}
  \lim_{t_1,t_2\rightarrow \infty} \parenthese{ \|u\|_{L_t^{l^{'}}L_x^{m^{'}}L^2_{\alpha}(t_1<t<t_2)}+\||\nabla_x|^2(u)\|_{L_t^{l}L_x^{m}L^2_{\alpha}(t_1<t<t_2)}+\||\partial_{\alpha}|^2(u)\|_{L_t^{l}L_x^{m}L^2_{\alpha}(t_1<t<t_2)}}< \infty. 
\end{equation}

The above spacetime bounds are sufficient to show the scattering for \eqref{main}. In this step, all spacetime norms are over $\mathbb{R}_t\times \R^d \times \T$ unless indicated otherwise. For example, we denote $(\int_{t_0}^{\infty}|f(t)|^{p}dt)^{\frac{1}{p}}$ by $\|f(t)\|_{L^p_{t>t_0}}$ for any given time-dependent function $f(t)$. We note that we will apply a $H_{\alpha}^{\frac{1}{2}+\delta}$ valued version of the critical analysis of \cite{cazenave2003semilinear}.

\emph{Proof.} Using Strichartz estimates and the H\"older inequality,
\begin{equation}
\aligned
    \|u\|_{L_{t>t_0}^{q_{\theta}}L_x^{r_{\theta}}H_{\alpha}^{\frac{1}{2}+\delta}} &\lesssim \|u_0\|_{H^2 (\R^d \times \T)}+\||u|^pu\|_{L_{t>t_0}^{\tilde{q_{\theta}}^{'}}L_x^{\tilde{r_{\theta}}^{'}}H_{\alpha}^{\frac{1}{2}+\delta}} \\
    &\lesssim \|u_0\|_{H^2(\R^d \times \T)}+\|u\|^{1+p}_{L_{t>t_0}^{(1+p)\tilde{q_{\theta}}^{'}}L_x^{(1+p)\tilde{r_{\theta}}^{'}}H_{\alpha}^{\frac{1}{2}+\delta}} \\
    &\lesssim \|u_0\|_{H^2(\R^d \times \T)}+\|u\|^{(1+p)\theta}_{L_{t>t_0}^{q_{\theta}}L_x^{r_{\theta}}H_{\alpha}^{\frac{1}{2}+\delta}}\|u\|^{(1+p)(1-\theta)}_{L_{t>t_0}^{\infty}L_x^{\frac{pd}{2}}H_{\alpha}^{\frac{1}{2}+\delta}}.
\endaligned    
\end{equation}
Similar to Lemma 2.5 in \cite{TV2} (this lemma is an analysis result which does not involve the nonlinear PDE structure so we can use it directly), based on the decay property \eqref{decay}, we can further obtain
\begin{equation}\label{decay2}
 \|u\|_{L_x^{\frac{pd}{2}}H_{\alpha}^{\frac{1}{2}+\delta}}=o(1).
\end{equation}
Using the decay property \eqref{decay2}, we see for every $\epsilon>0$ there exists $t_0=t_0(\epsilon)>0$ such that
\begin{equation}
    \|u\|_{L_{t>t_0}^{q_{\theta}}L_x^{r_{\theta}}H_{\alpha}^{\frac{1}{2}+\delta}} \leq C\|u_0\|_{H^2(\R^d \times \T)}+\epsilon \|u\|_{L_{t>t_0}^{q_{\theta}}L_x^{r_{\theta}}H_{\alpha}^{\frac{1}{2}+\delta}}.
\end{equation}
We can now use the continuity argument to obtain
\begin{equation}
\|u\|_{L_{t>0}^{q_{\theta}}L_x^{r_{\theta}}H_{\alpha}^{\frac{1}{2}+\delta}}<\infty,    
\end{equation}
Similarly, we obtain $\|u\|_{L_{t<0}^{q_{\theta}}L_x^{r_{\theta}}H_{\alpha}^{\frac{1}{2}+\delta}}<\infty $.

Now we consider the second estimate. We show $\||\partial_{\alpha}|^2(u)\|_{L_t^{l}L_x^{m}L^2_{\alpha}}$, the other estimates are similar.  Using Strichartz estimate and the H\"older inequality,
\begin{equation}
\aligned
   \||\partial_{\alpha}|^2(u)\|_{L_{t>t_0}^{l}L_x^{m}L^2_{\alpha}} &\lesssim \|u_0\|_{H^2(\R^d \times \T)}+\||\partial_{\alpha}|^2(|u|^pu)\|_{L_{t>t_0}^{l^{'}}L_x^{m^{'}}L^2_{\alpha}} \\
    &\lesssim \|u_0\|_{H^2(\R^d \times \T)}+\||\partial_{\alpha}|^2(u)\|_{L_{t>t_0}^{l}L_x^{m}L^2_{\alpha}}\|u\|^p_{L_{t>t_0}^{q_{\theta}}L^{r_{\theta}}_xL_{\alpha}^{\infty}} \\
    &\lesssim \|u_0\|_{H^2(\R^d \times \T)}+\||\partial_{\alpha}|^2(u)\|_{L_{t>t_0}^{l}L_x^{m}L^2_{\alpha}}\|u\|^p_{L_{t>t_0}^{q_{\theta}}L^{r_{\theta}}_xH_{\alpha}^{\frac{1}{2}+\delta}}.
\endaligned    
\end{equation}
We conclude by choosing $t_0$ large enough and by recalling \eqref{est1}.

For 4NLS, for the sake of Strichartz estimates and Sobolev embedding, we choose the indices satisfying (see Lemma \ref{index2}), 
\begin{equation}
   s+\frac{1}{2}+\delta \leq 2, 
\end{equation}
\begin{equation}
    \frac{4}{q_{\theta}}+\frac{d}{r_{\theta}}=\frac{d}{2}-s,\frac{4}{q_{\theta}}+\frac{d}{\tilde{r}_{\theta}}+\frac{4}{\tilde{q}_{\theta}}+\frac{d}{r_{\theta}}=d,
\end{equation}
\begin{equation}
    \frac{1}{(p+1)\tilde{q_{\theta}}^{'}}=\frac{\theta}{q_{\theta}}, \frac{1}{(p+1)\tilde{r_{\theta}}^{'}}=\frac{\theta}{r_{\theta}}+\frac{2(1-\theta)}{pd},
\end{equation}
and
\begin{equation}
    \frac{4}{l}+\frac{d}{m}=\frac{d}{2},\frac{1}{m^{'}}=\frac{1}{m}+\frac{p}{r_{\theta}},\frac{1}{l^{'}}=\frac{1}{l}+\frac{p}{q_{\theta}}.
\end{equation}
\subsection{Step 3: Proof of the scattering asymptotics}
In fact by using the integral equation, it is sufficient to prove that
\begin{equation}
 \lim_{t_1,t_2\rightarrow \infty}\norm{\int_{t_1}^{t_2}e^{-is\Delta_{x,\alpha}^2}(|u|^pu) \, ds  }_{H^2_{x,\alpha}  (\R^d \times \T)}=0   
\end{equation}
Moreover, using Strichartz estimates, we only need to show,
\begin{equation}
  \lim_{t_1,t_2\rightarrow \infty}\parenthese{ \||u|^pu\|_{L_t^{l^{'}}L_x^{m^{'}}L^2_{\alpha}([t_1,t_2]\times \mathbb{R}^d \times \mathbb{T})}+\||\nabla_x|^2(|u|^pu)\|_{L_t^{l^{'}}L_x^{m^{'}}L^2_{\alpha}([t_1,t_2]\times \mathbb{R}^d \times \mathbb{T})}+\||\partial_{\alpha}|^2(|u|^pu)\|_{L_t^{l^{'}}L_x^{m^{'}}L^2_{\alpha}([t_1,t_2]\times \mathbb{R}^d \times \mathbb{T})}}=0. 
\end{equation}
Noticing the two established estimates, the above limit follows in a straightforward way. Thus we proved scattering in the energy space.

\section{Discussions regarding the higher dimensional analogue and the focusing scenario}\label{sec Discussion}
In this section, we discuss the higher dimensional analogue (4NLS on $\mathbb{R}^d\times \mathbb{T}^n$, $n=2,3$) and the focusing scenario. We will explain how they follow respectively.

\subsection{The higher dimensional analogue}
As mentioned in the introduction, we can deal with more general settings, i.e. on waveguides with more torus directions. Here are the models that we consider:
\begin{itemize}
\item
For two torus dimensions, $d\geq 5$ and $\frac{8}{d}<p<\frac{8}{d-2}$.
\begin{equation*}
    (i\partial_t+\Delta_{x,\alpha}^{2})u+|u|^pu=0,\quad u(0)= u_0(x,\alpha)\in H^2(\mathbb{R}^d\times \mathbb{T}^2) ;
\end{equation*}

\item
For three torus dimensions, $d\geq 5$ and $\frac{8}{d}<p<\frac{8}{d-1}$
\begin{equation*}
    (i\partial_t+\Delta_{x,\alpha}^{2})u+|u|^pu=0,\quad u(0)= u_0(x,\alpha)\in H^2(\mathbb{R}^d\times \mathbb{T}^3) .
\end{equation*}
\end{itemize}

Similar to \eqref{main}, we note that both of the models are of `subcritical natural'. Since the proofs are similar to the $\mathbb{R}^d\times \mathbb{T}$ case, we only mention the key propositions and skip the proofs.\vspace{3mm}

For Strichartz estimates in the setting of $\mathbb{R}^d\times \mathbb{T}^n$, one can prove more general version as follows
\begin{lemma}
We consider B-admissible Strichartz pair $(p,q)$ with $s$ regularity  ($\frac{4}{p}+\frac{d}{q}=\frac{d}{2}-s$) and $\gamma>0$,
\begin{equation}
   \|e^{it\Delta_{x,\alpha}^{2}} u_0\|_{L_t^pL_x^qH_y^{\gamma}(\mathbb{R} \times \mathbb{R}^d\times \mathbb{T}^n)} \lesssim \|  u_0\|_{H_x^sH_y^{\gamma}(\mathbb{R}^d\times \mathbb{T}^n)}
\end{equation}
and
\begin{equation}
    \norm{\int_0^{t}e^{i(t-s)\Delta_{x,\alpha}^{2}} F(s)ds}_{L_t^pL_x^qH_y^{\gamma}(\mathbb{R} \times \mathbb{R}^d\times \mathbb{T}^n)} \lesssim \| |\nabla|^s F\|_{L_{t\in I}^{a^{'}}L_x^{b^{'}}H_y^{\gamma}(\mathbb{R} \times \mathbb{R}^d\times \mathbb{T}^n)}.
\end{equation}

\end{lemma}
Applying Strichartz estimates, one can establish the well-posedness theory via contraction mapping method as the $\mathbb{R}^d\times \mathbb{T}$ setting (see Section \ref{sec LWP}).\vspace{3mm}

For Morawetz estimates in the setting of $\mathbb{R}^d\times \mathbb{T}^n$, the dimension of torus direction does not make any difference (as shown in Remark \ref{rmk: Mora}) and one can prove,
\begin{proposition}[Morawetz estimates on $\R^d \times \T^n$]
Let $u (t,x,\alpha) \in C(\R; H_{x,\alpha}^2)$ be a global solution to the following defocusing 4NLS posed on $\R^d \times \T^n$, with $\frac{8}{d} < p < \frac{8}{d+n-4}$, $d \geq 5$
\begin{align*}
\begin{cases}
i \partial_t u + \Delta_{x,\alpha}^2 u + \abs{u}^p u = 0 \\
u(0,x) = u_0 .
\end{cases}
\end{align*}
Then  for every $r >0$, there exists $C = C(r)$ such that
\begin{align*}
\int_{\R} \parenthese{\sup_{x_0 \in \R^d} \iint_{Q^d(x_0 , r) \times \T^n} \abs{ u(t,x,\alpha)}^2 \, dx d\alpha}^{\frac{p+4}{2}} \, dt  \leq C  \norm{u_0}_{H_{x, \alpha}^2 (\R^d \times \T^n)}^4 .
\end{align*}
\end{proposition}
Following the scheme of the $\mathbb{R}^d\times \mathbb{T}$ case (see Section 5), we can obtain the decay property then the scattering behavior. (See also the Appendix for the decay property of 4NLS on general waveguide manifolds.)\vspace{3mm}

At last, we \emph{emphasize} that using the methods in this paper we can at most deal with the case of 3-dimensional torus component. If one considers 4NLS on $\mathbb{R}^d \times \mathbb{T}^4$, the `double subcritical nature' will break up thus all subcritical techniques fail. We also note that, if one 4NLS on $\mathbb{R}^d \times \mathbb{T}^n$ ($n\geq 5$), scattering behavior will not be expected though one may still consider proving the global well-posedness. As a comparison, heuristically, 4NLS on $\mathbb{R}^d\times \mathbb{T}^n$ ($n=1,2,3$) is like NLS on $\mathbb{R}^d\times \mathbb{T}$; 4NLS on $\mathbb{R}^d\times \mathbb{T}^4$ is like NLS on $\mathbb{R}^d\times \mathbb{T}^2$; and 4NLS on $\mathbb{R}^d\times \mathbb{T}^n$ ($n\geq 5$) is like NLS on $\mathbb{R}^d\times \mathbb{T}^n$ ($n \geq 2$).
\subsection{The focusing scenario}
We can also prove global well-posedness for the focusing scenario (mass-subcritical case). However, to prove global well-posedness and scattering for general settings (at least mass critical), more ingredients are required (threshold assumptions are heuristically needed), see Dodson \cite{dodson2019global}, Killip-Visan \cite{killip2010focusing}, Duyckaerts-Holmer-Roudenko \cite{DHR} and Kenig-Merle \cite{kenig2006global} for examples.

We consider the following model,
\begin{equation}\label{focmaineq}
    (i\partial_t+\Delta_{x,\alpha}^{2})u-|u|^pu=0,\quad u(0)= u_0(x,\alpha)\in H^2(\mathbb{R}^d\times \mathbb{T}^n),
\end{equation}
with $(t,x,\alpha)\in \mathbb{R}_t \times \mathbb{R}^d \times \mathbb{T}^{n}$, where $d\geq 5$ and $0<p<\frac{8}{d}$.\vspace{3mm}

We note that $0<p<\frac{8}{d}$ indicates the mass critical case and we do not constrain the dimension of the torus component.\vspace{3mm}

As for the local theory of \eqref{focmaineq}, the focusing scenario has no differences from the defocusing one. From the local to the global, it suffices to show the $H^2$-norm does not blow up.

For the focusing case, we consider the following conserved quantity which combines both of the mass and energy,
\begin{equation}
\int_{\mathbb{R}^d\times \mathbb{T}^{n}} |u(t,x,\alpha)|^2\, d x d\alpha+\int_{\mathbb{R}^d\times \mathbb{T}^{n}} \frac12 |\Delta_{x,\alpha} u(t,x,\alpha)|^2  - \frac{1}{p+2} |u(t,x,\alpha)|^{p+2} \, dx d\alpha.   
\end{equation}
By the Gagliardo–Nirenberg inequality we deduce
\begin{equation}
    \aligned
    & \int_{\mathbb{R}^d\times \mathbb{T}^{n}} |u_0|^2\, d x d\alpha+\int_{\mathbb{R}^d\times \mathbb{T}^{n}} \frac12 |\Delta_{x,\alpha} u(0,x,\alpha)|^2  - \frac{1}{p+2} |u(0,x,\alpha)|^{p+2} \, dx d\alpha \\
    &= \int_{\mathbb{R}^d\times \mathbb{T}^{n}} |u(t,x,\alpha)|^2\, d x d\alpha+\int_{\mathbb{R}^d\times \mathbb{T}^{n}} \frac12 |\Delta_{x,\alpha} u(t,x,\alpha)|^2  - \frac{1}{p+2} |u(t,x,\alpha)|^{p+2} \, dx d\alpha \\
    &\geq \frac{1}{2}\|u(t)\|_{H^2(\mathbb{R}^d\times \mathbb{T}^{n})}-C\|u_0\|^{2+p-\theta}_{L^2_{x,\alpha}(\mathbb{R}^d\times \mathbb{T}^{n})}\|u(t)\|^{\theta}_{H^1(\mathbb{R}^d\times \mathbb{T}^{n})} \\
    &\geq  \frac{1}{2}\|u(t)\|_{H^2(\mathbb{R}^d\times \mathbb{T}^{n})}-C\|u_0\|^{2+p-\theta}_{L^2_{x,\alpha}(\mathbb{R}^d\times \mathbb{T}^{n})}\|u(t)\|^{\theta}_{H^2(\mathbb{R}^d\times \mathbb{T}^{n})},  \\
    \endaligned
\end{equation}
where $\theta \in (0,2)$ since the problem is mass-subcritical. It implies that the $H^2$-norm of the solution cannot blow up in finite time.

\section{Further remarks}\label{sec Remarks}
In this section, we make a few more remarks on this research line, i.e. \emph{long time dynamics for dispersive equations on waveguide manifolds}. As mentioned in the introduction, this area has been developed a lot in recent decades. Though many theories/tools/results have been established, there are still many interesting open questions left. We list some interesting related problems in this line for interested readers.\vspace{3mm}

1. \emph{The critical regime.} The cases we are considering in this paper are of `double subcritical' nature (see \eqref{maineq}, \eqref{maineq2}). In fact, it is also quite interesting to consider the scattering theory for the critical regime. For example, 
\begin{equation}
    (i\partial_t+\Delta_{x,\alpha}^{2})u+|u|^{\frac{8}{d}}u=0,\quad u(0)= u_0(x,\alpha)\in H^2(\mathbb{R}^d\times \mathbb{T}),
\end{equation}
and
\begin{equation}
    (i\partial_t+\Delta_{x,\alpha}^{2})u+|u|^{\frac{8}{d-3}}u=0,\quad u(0)= u_0(x,\alpha)\in H^2(\mathbb{R}^d\times \mathbb{T}),
\end{equation}
The first one is of mass-critical nature and the second one is of energy critical nature. New techniques are needed including function spaces, profile decomposition, profile approximations and even resonant systems. See \cite{CGZ,HP,Z1} for the NLS case.\vspace{3mm}

2. \emph{More general waveguide settings.} As mentioned in the introduction, due to the technical restrains, we can at most do $\mathbb{R}^d\times \mathbb{T}^3$. It is interesting to consider the general case $\mathbb{R}^d\times \mathbb{T}^n$. For more than four dimensions of tori, scattering is not expected but one may still study the long time behavior, global well-posedness for example. \vspace{3mm}

3. \emph{Scattering for focusing NLS/4NLS on waveguide manifolds.} For the focusing scenario, there is no scattering results for NLS/4NLS on waveguide manifolds, to the best acknowledgement of the authors. New ingredients are needed to deal with this type of problems and threshold assumptions are necessary. See \cite{YYZ} for a recent global well-posedness result. And see \cite{dodson2019global,DHR,kenig2006global,killip2010focusing} for the Euclidean result. \vspace{3mm}

4. \emph{Critical NLS on higher dimensional waveguide manifolds.} For critical NLS on waveguide manifolds, most of the models are lower dimensional (with no bigger than four whole dimensions), which leads quintic or cubic nonlinearity. This gives one advantages to apply function spaces to deal with the nonlinearity. In general, the difficulty of the critical NLS problem on $\mathbb{R}^m \times \mathbb{T}^n$ increases if the dimension $m+n$ is increased or if the number $m$ of copies of $\mathbb{R}$ is decreased (It is concluded in \cite{IPRT3}). There is no large data global results for critical NLS on waveguide manifolds with at least $5$ whole dimensions, to the best acknowledgement of the authors. \vspace{3mm}

5. \emph{NLS on other product spaces.} Instead of waveguide manifolds, one may consider dispersive equations on other types of product spaces. For example, $\mathbb{R}^d\times \mathbb{S}^n$ where $\mathbb{S}^n$ are n-dimensional spheres ($\mathbb{S}^n$ can be replaced by other manifolds). See \cite{pausader2014global} for a Global well-posedness result of NLS on pure spheres. NLS may be a good model to start with.\vspace{3mm}

\appendix
\section{Decay property for NLS and 4NLS on waveguides}\label{sec Apx}
In this appendix, we include the decay property for NLS and 4NLS on general waveguide manifolds $\mathbb{R}^d\times \mathbb{T}^n$ via the interaction Morawetz estimate and the contradiction argument, which may have their own interests. They are mainly motivated by \cite{terracini2014nonlinear,TV2,visciglia2008decay}. See also \cite{visciglia2008decay} for the Euclidean setting. 

As a comparison, we include the decay property results in the Euclidean setting below. 

\begin{proposition}[\cite{visciglia2008decay}]
Let $u(t,x) \in C(\mathbb{R};H^1(\mathbb{R}^d))$ be the unique global solution to
\begin{equation}
 (i\partial_t+\Delta_x)u=|u|^p u, \quad u(0)=u_0 (x) \in H^1(\mathbb{R}^d)  
\end{equation}
where $0<p<\frac{4}{d-2}$. Then for every $2<r<\frac{2d}{d-2}$ when $d=3$ and for every $2<r<\infty$ when $d=1,2$, we have:
\begin{equation}
    \lim_{t \rightarrow \pm \infty}\|u(t,x)\|_{L_{x}^{r}(\mathbb{R}^d)}=0.
\end{equation}
Moreover in the case $d=1$ we also have:
\begin{equation}
  \lim_{t \rightarrow \pm \infty}\|u(t,x)\|_{L_{x}^{\infty}(\mathbb{R}^d)}=0.   
\end{equation}
\end{proposition}

The results for the waveguide case are as follows. We discuss the 4NLS case and the NLS case respectively.\vspace{3mm}

\subsection{Decay property for 4NLS on waveguides}

We consider the 4NLS on waveguides $\mathbb{R}^d\times \mathbb{T}^n$ ($d\geq 5, n \geq 1$) as follows.
\begin{equation}
    (i\partial_t+\Delta_{x,\alpha}^{2})u+|u|^pu=0,\quad u(0)= u_0(x,\alpha)\in H^2(\mathbb{R}^d\times \mathbb{T}^n).
\end{equation}
We will show 
\begin{equation}
 \lim_{t \rightarrow \infty} \|u(t,x,\alpha)\|_{L^q_{x,\alpha} (\R^d \times \T^n)}=0,   
\end{equation}
where $2<q \leq 2+\frac{4}{d+n}$.

We recall the Morawetz estimate (see Proposition \ref{Morawetz} and Remark \ref{rmk: Mora}),
\begin{equation}
    \int_{\mathbb{R}} \parenthese{ \sup_{x_0\in \mathbb{R}^d} \iint_{Q^d(x_0,r)\times \mathbb{T}^n} |u(t,x,\alpha)|^2 \, dxd\alpha}^{\frac{p+4}{2}} \, dt \lesssim \|u_0\|^4_{H^{2}_{x,\alpha} (\R^d \times \T^n)}.
\end{equation}
We then recall the localized Gagliardo–Nirenberg inequality,
\begin{equation}
    \|v\|_{L_{x,\alpha}^{2+\frac{4}{d+n}} (\R^d \times \T^n)}\lesssim \sup_{x\in \mathbb{R}^d} \parenthese{ \|v\|^{2}_{L^2_{Q^{d}_{x}\times \mathbb{T}^n}}}^{\frac{2}{d+n+2}}\|v\|^{\frac{d+n}{d+n+2}}_{H^1(\mathbb{R}^d\times \mathbb{T}^n)}
\end{equation}
where $Q^{d}_{x}=x+[0,1]^d$ for all $x\in \mathbb{R}^d$.

Now we prove the decay property via a contradiction argument. Obviously it is sufficient to show that
\begin{equation}
    \lim_{t\rightarrow \infty}\|u(t,x,\alpha)\|_{L_{x,\alpha}^{2+\frac{4}{d+n}} (\R^d \times \T^n)} =0,
\end{equation}
since other bounds can be obtained via interpolation with the conservation law.

Next, assume the decay estimate does not hold. Then we deduce the existence of a sequence $(t_n, x_n) \in \mathbb{R}\times \mathbb{R}^d$ with $t_n \rightarrow \infty$ and $\varepsilon_0>0$
such that
\begin{equation}
    \inf_{n} \|u(t_n,x,\alpha)\|_{L^2_{Q^{d}(x_n,1)\times \mathbb{T}^n}} =\varepsilon_0.
\end{equation}
We get the existence of $T > 0$ such that
\begin{equation}
    \inf_{n} \parenthese{ \inf_{t\in (t_n,t_n+T)}\|u(t,x,\alpha)\|_{L^2_{Q^{d}(x_n,2)\times \mathbb{T}^n}} } \geq \frac{\varepsilon_0}{2}.
\end{equation}
Notice that since $t_n \rightarrow \infty$ then we can assume (up to a subsequence) that the
intervals $(t_n,t_n+T)$ are disjoint. In particular we have
\begin{equation}
\aligned
\sum_{n} T(\frac{\varepsilon_0}{2})^{p+4}&\leq \sum_{n}\int_{t_n}^{t_n+T} \parenthese{  \iint_{Q^d(x_n,2)\times \mathbb{T}^n} |u(t,x,\alpha)|^2\, dxd\alpha }^{\frac{p+4}{2}}\, dt\\
&\leq \int_{\mathbb{R}} \parenthese{ \sup_{z\in \mathbb{R}^d} \iint_{Q^d(z,2)\times \mathbb{T}^n} |u(t,x,\alpha)|^2 \, dxd\alpha }^{\frac{p+4}{2}} \, dt
\endaligned
\end{equation}
and hence we get a contradiction since the left hand side is divergent and the right hand side is bounded by the Morawetz estimate.
\subsection{Decay property for NLS on waveguides}
We consider NLS on waveguides $\mathbb{R}^d\times \mathbb{T}^n$ ($d, n \geq 1$) as follows.
\begin{equation}
    (i\partial_t+\Delta_{x,\alpha})u-|u|^pu=0,\quad u(0)= u_0(x,\alpha)\in H^1(\mathbb{R}^d\times \mathbb{T}^n).
\end{equation}
We will show 
\begin{equation}
 \lim_{t \rightarrow \infty} \|u(t,x,\alpha)\|_{L^q_{x,\alpha} (\R^d \times \T^n)}=0,   
\end{equation}
where $2<q \leq 2+\frac{4}{d+n}$.

We recall the Morawetz estimate (see Proposition \ref{Morawetz} and Remark \ref{rmk: Mora}),
\begin{equation}
    \int_{\mathbb{R}} \parenthese{ \sup_{x_0\in \mathbb{R}^d} \iint_{Q^d(x_0,r)\times \mathbb{T}^n} |u(t,x,\alpha)|^2 \, dxd\alpha}^{\frac{p+4}{2}} \, dt \lesssim \|u_0\|^4_{H^{1}_{x,\alpha} (\R^d \times \T^n)}.
\end{equation}
Recalling the localized Gagliardo–Nirenberg inequality,
\begin{equation}
    \|v\|_{L_{x,\alpha}^{2+\frac{4}{d+n}} (\R^d \times \T^n)}\lesssim \sup_{x\in \mathbb{R}^d} \parenthese{\|v\|^{2}_{L^2_{Q^{d}_{x}\times \mathbb{T}^n}}}^{\frac{2}{d+n+2}}\|v\|^{\frac{d+n}{d+n+2}}_{H^1(\mathbb{R}^d\times \mathbb{T}^n)}
\end{equation}
where $Q^{d}_{x}=x+[0,1]^d$ for all $x\in \mathbb{R}^d$.

Now we prove the decay property via a contradiction argument. Obviously it is sufficient to show that
\begin{equation}
    \lim_{t\rightarrow \infty}\|u(t,x,\alpha)\|_{L_{x,\alpha}^{2+\frac{4}{d+n}} (\R^d \times \T^n)} =0,
\end{equation}
since other bounds can be obtained via interpolation with the conservation law.

Next, assume the decay estimate does not hold. Then we deduce the existence of a sequence $(t_n, x_n) \in \mathbb{R}\times \mathbb{R}^d$ with $t_n \rightarrow \infty$ and $\varepsilon_0>0$
such that
\begin{equation}
    \inf_{n} \|u(t_n,x,\alpha)\|_{L^2_{Q^{d}(x_n,1)\times \mathbb{T}^n}} =\varepsilon_0.
\end{equation}
We get the existence of $T > 0$ such that
\begin{equation}
    \inf_{n} \parenthese{ \inf_{t\in (t_n,t_n+T)}\|u(t,x,\alpha)\|_{L^2_{Q^{d}(x_n,2)\times \mathbb{T}^n}} } \geq \frac{\varepsilon_0}{2}.
\end{equation}
Notice that since $t_n \rightarrow \infty$ then we can assume (up to a  subsequence) that the
intervals $(t_n,t_n+T)$ are disjoint. In particular we have
\begin{equation}
\aligned
\sum_{n} T(\frac{\varepsilon_0}{2})^{p+4}&\leq \sum_{n}\int_{t_n}^{t_n+T} \parenthese{ \iint_{Q^d(x_n,2)\times \mathbb{T}^n} |u(t,x,\alpha)|^2 \, dxd\alpha}^{\frac{p+4}{2}} \, dt\\
&\leq \int_{\mathbb{R}} \parenthese{ \sup_{z\in \mathbb{R}^d} \iint_{Q^d(z,2)\times \mathbb{T}^n} |u(t,x,\alpha)|^2 \, dxd\alpha}^{\frac{p+4}{2}} \, dt
\endaligned
\end{equation}
and hence we get a contradiction since the left hand side is divergent and the right hand side is bounded by the Morawetz estimate.\vspace{3mm}

At last, we conclude this paper by mentioning two more remarks. We discuss the NLS case here and similar statements can be made for  4NLS.

\begin{remark}
As we can see from the discussions above, for the $\mathbb{R}^d \times \mathbb{T}^n$ waveguide case, one can obtain decay for $L^q_{x,\alpha}$-norm where $q$ is at most $2+\frac{4}{d+n}$. It is interesting to think if $2+\frac{4}{d+n}$ is the maximal index. If not, can one obtain larger range?
\end{remark}

\begin{remark}
One may also consider the pointwise type decay which describes the decay rate of nonlinear solutions quantitatively. Heuristically, for NLS on waveguides $\mathbb{R}^d \times \mathbb{T}^n$ with proper nonlinearity (in the sense that the scattering holds) and nice initial space, the nonlinear solution in $L^{\infty}$-norm decays as $t^{-\frac{d}{2}}$, which is consistent with the Euclidean results. See \cite{fan2021decay} for a recent result and the references therein.
\end{remark}
\noindent \textbf{Acknowledgments.} X. Yu was funded in part by an AMS-Simons travel grant. H. Yue was supported by a start-up funding of ShanghaiTech University. Z. Zhao was supported by UMD’s postdoc support, NSFC-12101046 and the Beijing Institute of Technology Research Fund Program for Young Scholars. Part of this work was done while the first two authors were in residence at the Institute for Computational and Experimental Research in Mathematics in Providence, RI, during the {\it Hamiltonian Methods in Dispersive and Wave Evolution Equations} program. Some of the work was done while the third author was moving from the University of Maryland to the Beijing Institute of Technology, so he appreciates the kind supports of both institutes. \vspace{5mm}

\bibliography{4NLS}
\bibliographystyle{plain}
\end{document}